\documentclass[11pt,reqno]{amsart}
\usepackage{amsmath}
\usepackage{amsthm,amscd,graphicx,float,stmaryrd,youngtab,mathtools}
\usepackage{dsfont}
\usepackage{amssymb}
\usepackage{amsfonts}
\usepackage{mathrsfs}
\usepackage{mdframed}
\usepackage{colortbl}
\usepackage{array}
\usepackage{bm}

\usepackage{tikz}
\usepackage{tikz-cd}
\usetikzlibrary{calc,matrix}
\usepackage{enumitem}
\usepackage[all]{xy}
\usetikzlibrary{decorations.pathmorphing,decorations.markings,decorations.pathreplacing,arrows,shapes}
\usetikzlibrary{knots}

\usepackage[bookmarks=true, bookmarksopen=true,%
bookmarksdepth=3,bookmarksopenlevel=2,%
colorlinks=true,%
linkcolor=blue,%
citecolor=blue,%
filecolor=blue,%
menucolor=blue,%
urlcolor=blue]{hyperref}

\newtheorem{thm}{Theorem}[section]
\newtheorem{prop}[thm]{Proposition}
\newtheorem{cor}[thm]{Corollary}
\newtheorem{lem}[thm]{Lemma}

\newtheorem{Mthm}[thm]{Main Theorem}

\theoremstyle{definition}
\newtheorem{defn}[thm]{Definition}

\newtheorem{rmk}[thm]{Remark}
\newtheorem{exmp}[thm]{Example}

\newtheorem{assumption}[thm]{Assumption}

\makeatletter
\let\c@equation\c@thm
\makeatother
\numberwithin{equation}{section}

\newcommand{\conf}{\mathrm{Conf}}

\newcommand{\Gr}{\mathrm{Gr}}

\newcommand{\SL}{\mathrm{SL}}

\newcommand{\DT}{\mathrm{DT}}

\newcommand{\Aug}{\mathrm{Aug}}

\newcommand{\G}{\mathsf{G}}

\newcommand{\R}{\mathsf{R}}

\newcommand{\K}{\mathsf{K}}

\newcommand{\Br}{\mathsf{Br}}

\newcommand{\cC}{\mathcal{C}}

\newcommand{\cR}{\mathcal{R}}

\newcommand{\bbR}{\mathbb{R}}

\newcommand{\bbZ}{\mathbb{Z}}

\setenumerate[1]{label={(\arabic*)}}
\setlist[enumerate]{itemsep=0mm}

\begin{document}

\title[Positive Braid Links with Infinitely Many Fillings]{Positive Braid Links with Infinitely Many Fillings}
\author{Honghao Gao, Linhui Shen, and Daping Weng}
\date{}
\address{Honghao Gao \\ Department of Mathematics \\ Michigan State University}
\email{\href{mailto:gaohongh@msu.edu}{gaohongh@msu.edu}}

\address{Linhui Shen \\ Department of Mathematics \\ Michigan State University}
\email{\href{mailto:linhui@math.msu.edu}{linhui@math.msu.edu}}

\address{Daping Weng \\ Department of Mathematics \\ Michigan State University}
\email{\href{mailto:wengdap1@msu.edu}{wengdap1@msu.edu}}

\maketitle

\begin{abstract}
    We prove that any positive braid Legendrian link not isotopic to a standard finite type link admits infinitely many exact Lagrangian fillings.
\end{abstract}

\section{Introduction}
To this day, the only complete classification of exact Lagrangian fillings of a given Legendrian knot was the unique filling of the unknot with maximal Thurston-Bennequin number \cite{EP}. Most subsequent work focuses on giving a lower bound on the number of distinct fillings, which is typically achieved by constructing fillings explicitly and distinguishing them using an invariant. Known constructions of fillings include (1) decomposable Lagrangian fillings \cite{EHK12}, (2) conjugate Lagrangians for alternating Legendrians \cite{STWZ}, and (3) free Legendrian weaves \cite{TZ,CZ}. The invariants used to distinguish these fillings include augmentations \cite{EHK12, Pan_2017} and microlocal sheaves \cite{STWZ}. 

The question whether there exists a Legendrian link admitting infinitely many exact Lagrangian fillings remained open until the year 2020. Several methods emerged concurrently and each successfully solved the problem for a class of Legendrian links.

\begin{itemize}
    \item \cite{CasalsGao} Any torus $(n,m)$-link except $(2,m), (3,3), (3,4), (3,5)$ admits infinitely many Lagrangian fillings. The proof uses Legendrian loops, microlocal sheaves, and cluster algebras.
    
    \item \cite{CZ} The closure of the braid $\left[\beta_{p,q}\right]\in \Br_3^+$ admits infinitely many Lagrangian fillings, where
    $$\beta_{p,q}=(s_1^3s_2)(s_1^3s_2^2)^ps_1^3s_2(s_2^2s_1^3)^q(s_2s_1^3)(s_2^{q+1}s_1^2s_2^{p+2}) , \quad p,q\in \mathbb{N}_{\geq 1}.$$ 
    The proof uses Legendrian weaves, sheaves, and cluster algebras.
    
    \item The upcoming work \cite{CN} shows that certain Legendrian links (not necessarily positive braid links) have Legendrian loops of infinite order, using  Legendrian contact dga.
\end{itemize}

In this paper, we solve the infinitely many fillings problem for positive braid links using cluster structure on augmentation varieties \cite{GSW} and Donaldson-Thomas transformations of cluster varieties \cite{GS2, SWflag}.

\begin{Mthm}\label{Main} If a positive braid Legendrian link is not Legendrian isotopic to a split union of unknots and connect sums of standard $\mathrm{ADE}$ links, then it admits infinitely many non-Hamiltonian isotopic exact Lagrangian fillings.
\end{Mthm}

A $n$-strand  braid word $\beta$ is a finite sequence of the letters $s_1, \ldots, s_{n-1}$. Every braid word $\beta$ determines a quiver $Q_\beta$  by the following equivalent constructions: wiring diagrams \cite{FZ, BFZ}, amalgamation \cite{FGamalgamation}, and brick diagrams \cite{Rudolph,BLL}. In detail,  $Q_\beta$ can be constructed by the following steps. 

\begin{enumerate}
    \item Plot $\beta$ and replace each crossing by a vertical bar $\mathsf{I}$. We get a ``wall of bricks''.
    \item Draw a vertex at each compact brick.
    \item For any two adjacent bricks on the same level, draw a rightward horizontal arrow connecting them. For any two adjacent bricks forming a Z or S pattern, draw a leftward arrow connecting them.  
\end{enumerate}

A quiver is of \emph{infinite type} if it  has  infinitely many isomorphism classes of indecomposable representations; otherwise it is of \emph{finite type}. A \emph{Dynkin quiver} is a quiver whose underlying undirected graph is one of the Dynkin diagrams: $\mathrm{A}_r,\mathrm{D}_r, \mathrm{E}_6, \mathrm{E}_7, \mathrm{E}_8$. By a theorem of Gabriel \cite{Gab}, a quiver is of finite type if and only if it is a disjoint union of Dynkin quivers.

\begin{exmp}
The braid word $s_1^{r}s_2s_1^3s_2$ gives rise to an $\mathrm{E}_{r+3}$ quiver. For example, an $\mathrm{E}_9$ quiver is obtained as follows:

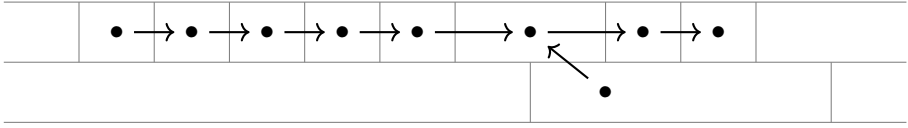
\[
\begin{tikzpicture}

\draw [gray] (0,0) -- (12,0);
\draw [gray] (0,0.8) -- (12,0.8);
\draw [gray] (0,1.6) -- (12,1.6);
\draw [gray] (1,0.8) -- (1, 1.6);
\draw [gray] (2,0.8) -- (2, 1.6);
\draw [gray] (3,0.8) -- (3, 1.6);
\draw [gray] (4,0.8) -- (4, 1.6);
\draw [gray] (5,0.8) -- (5, 1.6);
\draw [gray] (6,0.8) -- (6, 1.6);
\draw [gray] (7,0) -- (7, 0.8);
\draw [gray] (8,0.8) -- (8, 1.6);
\draw [gray] (9,0.8) -- (9, 1.6);
\draw [gray] (10,0.8) -- (10, 1.6);
\draw [gray] (11,0) -- (11, 0.8);

\node (1) at (1.5,1.2) [] {$\bullet$};
\node (2) at (2.5,1.2) [] {$\bullet$};
\node (3) at (3.5,1.2) [] {$\bullet$};
\node (4) at (4.5,1.2) [] {$\bullet$};
\node (5) at (5.5,1.2) [] {$\bullet$};
\node (6) at (8,0.4) [] {$\bullet$};
\node (7) at (7,1.2) [] {$\bullet$};
\node (8) at (8.5,1.2) [] {$\bullet$};
\node (9) at (9.5,1.2) [] {$\bullet$};

\draw [->, thick] (1) -- (2);
\draw [->, thick] (2) -- (3);
\draw [->, thick] (3) -- (4);
\draw [->, thick] (4) -- (5);
\draw [->, thick] (5) -- (7);
\draw [->, thick] (7) -- (8);
\draw [->, thick] (8) -- (9);
\draw [->, thick] (6) -- (7);
\end{tikzpicture}
\]
\end{exmp}

The equivalence class $[\beta]$ of braid words $\beta$ modulo the relations $s_is_js_i=s_js_is_j$ for $|i-j|=1$ and $s_is_j=s_js_i$ for $|i-j|\geq 2$ is a positive braid. The collection $\Br_n^+$ of $n$-strand positive braids under juxtaposition forms a semigroup. It is known that different braid words of the same braid yield mutation equivalent quivers (see \cite{FZI} for the definition of quiver mutations). We say a positive braid $[\beta]$ is of \emph{finite type} if its associated quiver $Q_\beta$ is mutation equivalent to a quiver of finite type; otherwise it is  of \emph{infinite type}.

The link closure of a positive braid $[\beta]$ admits a unique Legendrian representative $\Lambda_\beta$ with maximal Thurston-Bennequin number by the rainbow closure construction \cite{EV}. 

\begin{defn} \label{stad.links}
For each of the Dynkin diagrams, we define its \emph{standard link} to be the rainbow closures of the following positive braid words.
\begin{center}
    \begin{tabular}{|c|c|c|c|c|} \hline
    \rule{0pt}{2.3ex} $\Br_2^+$ & \multicolumn{4}{c|}{$\Br_3^+$} \\[0.045cm] \hline
    \rule{0pt}{2.3ex} $\mathrm{A}_r$     &  $\mathrm{D}_r$ & $\mathrm{E}_6$  & $\mathrm{E}_7$ & $\mathrm{E}_8$ \\[0.045cm] \hline
    \rule{0pt}{2.3ex} $s_1^{r+1}$ &     $ s_1^{r-2}s_2s_1^2s_2$ & $s_1^3s_2s_1^3s_2$ & $s_1^4s_2s_1^3s_2$ & $s_1^5s_2s_1^3s_2$ \\[0.045cm] \hline
    \end{tabular}
\end{center}
\end{defn}

\medskip

The infinitely many fillings problem refers to whether there exists a Legendrian link admitting infinitely many Hamiltonian-isotopic classes of exact Lagrangian fillings. We solve it by explicitly constructing infinitely many \emph{admissible fillings}. An {admissible} filling/cobordism/concordance is by definition a composition of  saddle cobordisms,  cyclic rotations, braid moves (R3), and  minimum cobordisms (\cite[Definition 1.3]{GSW}).

\smallskip

Here is a summary of several key results needed for the proof of Main Theorem \ref{Main}. 
In Section \ref{sec 2}, we prove that if $Q_\beta$ is acyclic and of infinite type, then $\Lambda_\beta$ admits infinitely many admissible fillings (Corollary \ref{rainbowinfinite}).
Building upon this result, we prove in Section \ref{sec 3} that if a positive braid $[\beta]$ is of infinite type, then $\Lambda_\beta$ admits infinitely many admissible fillings (Theorem \ref{MainStep1}).
In Section \ref{sec 4}, we prove that if a positive braid $[\beta]$ is finite type, then $\Lambda_\beta$ is Legendrian isotopic to a split union of unknots and connect sums of standard $\mathrm{ADE}$ links (Theorem \ref{4.6}). Main Theorem \ref{Main} follows from the dichotomy between finite and infinite types for positive braids. 

We discuss further applications in Section \ref{sec.5.app}.
\medskip

\noindent \textbf{Acknowledgement.} We thank everyone acknowledged in \cite{GSW} for their support on the overarching project. In particular, we would like to thank Roger Casals, Bernhard Keller, and Eric Zaslow for useful discussions and advice on references.

\section{DT Transformation and Infinitely Many Fillings} \label{sec 2}

Let $\Lambda_\beta$ be the rainbow closure Legendrian link associated with an $n$-strand braid word $\beta$, with a marked point near each right cusp. Let $\mathbb{F}$ denote an algebraically closed field of characteristic $2$ and let $\Aug(\Lambda_\beta)$ denote the $\mathbb{F}$-augmentation variety.

In \cite{GSW}, we construct cluster $\mathrm{K}_2$ structures on augmentation varieties $\Aug\left(\Lambda_\beta\right)$ and prove that admissible fillings of $\Lambda_\beta$ induce cluster seeds on $\Aug\left(\Lambda_\beta\right)$. Theorem 1.4 in \textit{loc,cit} states that if two admissible fillings induce distinct cluster seeds on $\Aug\left(\Lambda_\beta\right)$, then they are not Hamiltonian isotopic. In this section we will utilize this result to prove the existence of Legendrian links with infinitely many fillings.

\subsection{Full Cyclic Rotation}

Let $\beta=s_i\beta'$ be a braid word starting with the letter $s_i$. The cyclic rotation $\rho$ is a Legendrian isotopy from  $\Lambda_{s_i\beta'}$ to $\Lambda_{\beta's_i}$, illustrated by the following moves on the front projection of Legendrian links
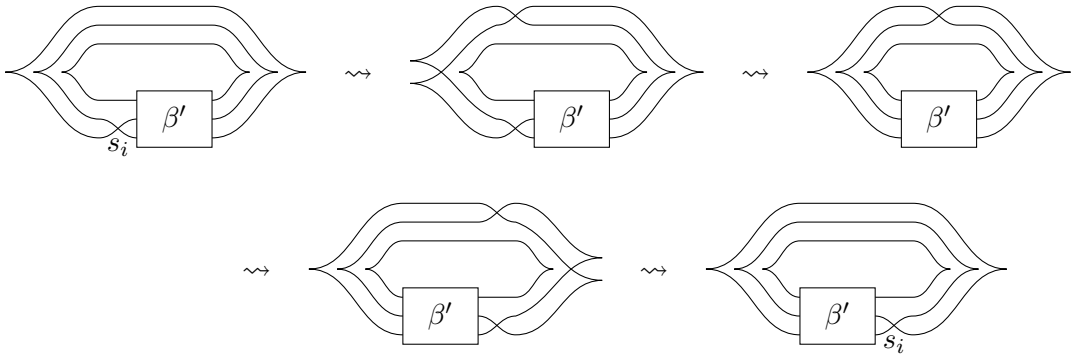
\[
\begin{tikzpicture}[baseline=25,scale =0.5]
    \draw (1,0) rectangle node [] {$\beta'$} (3,1.5);
    \draw  (3,0.25) to [out=0,in=180] (5.5,2) to [out=180,in=0] (3, 3.75) -- (0,3.75) to [out=180,in=0] (-2.5,2) to [out=0,in=180] (0,0.25) to [out=0,in=180] (1, 0.75);
    \draw (3,0.75) to [out=0,in=180] (4.75,2) to [out=180,in=0] (3, 3.25) -- (0,3.25) to [out=180,in=0] (-1.75,2) to [out=0,in=180] (0,0.75) to [out=0,in=180] (1, 0.25);
    \draw (3,1.25)  to [out=0,in=180] (4,2) to [out=180,in=0] (3, 2.75) -- (0,2.75) to [out=180,in=0] (-1,2) to [out=0,in=180] (0,1.25) -- (1,1.25);
    \node at (0.5, 0) [] {$s_i$};
\end{tikzpicture}
\quad \rightsquigarrow \quad
\begin{tikzpicture}[baseline=25,scale =0.5]
    \draw (1,0) rectangle node [] {$\beta'$} (3,1.5);
    \draw  (3,0.25) to [out=0,in=180] (5.5,2) to [out=180,in=0] (3, 3.75) -- (1,3.75) to [out=180,in=0] (0, 3.25) to [out=200,in=0] (-2.3,1.7) to [out=0,in=180] (0,0.25) to [out=0,in=180] (1, 0.75);
    \draw (3,0.75) to [out=0,in=180] (4.75,2) to [out=180,in=0] (3, 3.25) -- (1,3.25) to [out=180,in=0] (0, 3.75) to [out=180,in=0] (-2.3,2.3) to [out=0,in=160] (0,0.75) to [out=0,in=180] (1, 0.25);
    \draw (3,1.25)  to [out=0,in=180] (4,2) to [out=180,in=0] (3, 2.75) -- (0,2.75) to [out=180,in=0] (-1,2) to [out=0,in=180] (0,1.25) -- (1,1.25);
\end{tikzpicture}
\quad \rightsquigarrow \quad
\begin{tikzpicture}[baseline=25,scale =0.5]
    \draw (1,0) rectangle node [] {$\beta'$} (3,1.5);
    \draw  (3,0.25) to [out=0,in=180] (5.5,2) to [out=180,in=0] (3, 3.75) -- (2.5, 3.75) to [out=180,in=0] (1.5, 3.25) -- (1,3.25) to [out=180,in=0] (-0.75,2) to [out=0,in=180] (1,0.75);
    \draw (3,0.75) to [out=0,in=180] (4.75,2) to [out=180,in=0] (3, 3.25) -- (2.5, 3.25) to [out=180,in=0] (1.5, 3.75) -- (1,3.75) to [out=180,in=0] (-1.5,2) to [out=0,in=180] (1,0.25);
    \draw (3,1.25)  to [out=0,in=180] (4,2) to [out=180,in=0] (3, 2.75) -- (1,2.75) to [out=180,in=0] (0,2) to [out=0,in=180] (1,1.25);
\end{tikzpicture}
\]
\[
\qquad\qquad\qquad \rightsquigarrow \quad
\begin{tikzpicture}[baseline=25,scale =0.5]
    \draw (-1,0) rectangle node [] {$\beta'$} (-3,1.5);
    \draw  (-3,0.25) to [out=180,in=0] (-5.5,2) to [out=0,in=180] (-3, 3.75) -- (-1,3.75) to [out=0,in=180] (0, 3.25) to [out=-20,in=180] (2.3,1.7) to [out=180,in=0] (0,0.25) to [out=180,in=0] (-1, 0.75);
    \draw (-3,0.75) to [out=180,in=0] (-4.75,2) to [out=0,in=180] (-3, 3.25) -- (-1,3.25) to [out=0,in=180] (0, 3.75) to [out=0,in=180] (2.3,2.3) to [out=180,in=20] (0,0.75) to [out=180,in=0] (-1, 0.25);
    \draw (-3,1.25) to [out=180,in=0] (-4,2) to [out=0,in=180] (-3, 2.75) -- (0,2.75) to [out=0,in=180] (1,2) to [out=180,in=0] (0,1.25) -- (-1,1.25);
\end{tikzpicture}
\quad \rightsquigarrow \quad
\begin{tikzpicture}[baseline=25,scale =0.5]
    \draw (-1,0) rectangle node [] {$\beta'$} (-3,1.5);
    \draw  (-3,0.25) to [out=180,in=0] (-5.5,2) to [out=0,in=180] (-3, 3.75) -- (0,3.75) to [out=0,in=180] (2.5,2) to [out=180,in=0] (0,0.25) to [out=180,in=0] (-1, 0.75);
    \draw (-3,0.75) to [out=180,in=0] (-4.75,2) to [out=0,in=180] (-3, 3.25) -- (0,3.25) to [out=0,in=180] (1.75,2) to [out=180,in=0] (0,0.75) to [out=180,in=0] (-1, 0.25);
    \draw (-3,1.25)  to [out=180,in=0] (-4,2) to [out=0,in=180] (-3, 2.75) -- (0,2.75) to [out=0,in=180] (1,2) to [out=180,in=0] (0,1.25) -- (-1,1.25);
    \node at (-0.5, 0) [] {$s_i$};
\end{tikzpicture}
\]
 
The \emph{full cyclic rotation}  $\R$ of 
$\Lambda_\beta$ is the composition of cyclic rotations that rotate each letter in $\beta$ one-by-one from left to right. The resulted Legenedrian link is $\Lambda_beta$ again. Therefore $\R$ is a Legendrian loop, which induces an automorphism $\Phi_\R$ on $\Aug(\Lambda_\beta)$.

\begin{rmk} Below on the left is an example of the annular Legendrian weave (developed in \cite{CZ}) for the Lagrangian self concordance induced from the full cyclic rotation of a $3$-strand braid. The dashed teal region is the full twist $w_0^2$ coming from satelliting the positive braid along the unknot. This Legendrian $3$-graph consists mostly spirals, except they enter and leave the full twist region horizontally. The boxes on the right shows what happens inside that region.
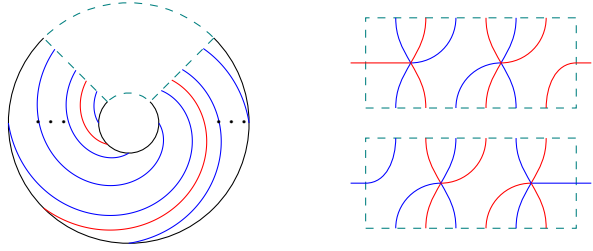
\[
\begin{tikzpicture}[baseline=0,scale=0.8]
\draw (45:2) arc (45:-225:2);
\draw (45:0.5) arc (45:-225:0.5);
\draw [teal, dashed] (45:2) -- (45:0.5) arc (45:135:0.5);
\draw [teal, dashed] (135:0.5) -- (135:2) arc (135:45:2);
\draw [blue,domain=0:45,variable=\t,smooth,samples=100] plot ({180-\t}: {0.5+0.0055*\t});
\draw [red,domain=0:90,variable=\t,smooth,samples=100] plot ({225-\t}: {0.5+0.0055*\t});
\draw [blue,domain=0:135,variable=\t,smooth,samples=100] plot ({270-\t}: {0.5+0.0055*\t});
\draw [blue,domain=0:225,variable=\t,smooth,samples=100] plot ({360-\t}: {0.5+0.0055*\t});
\draw [blue,domain=0:225,variable=\t,smooth,samples=100] plot ({180+\t}: {2-0.0055*\t});
\draw [red,domain=0:180,variable=\t,smooth,samples=100] plot ({225+\t}: {2-0.0055*\t});
\draw [blue,domain=0:135,variable=\t,smooth,samples=100] plot ({270+\t}: {2-0.0055*\t});
\draw [blue,domain=0:45,variable=\t,smooth,samples=100] plot ({0+\t}: {2-0.0055*\t});
\node at (180:1.25) [] {$\cdots$};
\node at (0:1.75) [] {$\cdots$};
\end{tikzpicture}
\quad \quad \quad 
\begin{tikzpicture}[baseline=0,scale=0.8]
\draw [teal, dashed] (0,0.25) rectangle (3.5,1.75);
\draw [teal, dashed] (0,-0.25) rectangle (3.5,-1.75);
\draw [red] (-0.25,1) -- (0.75,1);
\draw [blue] (0.75,1) to [out=0,in=-90] (1.5,1.75);
\draw [blue] (0.5,1.75) to [out=-90,in=120] (0.75,1);
\draw [blue] (-0.25,-1) -- (0,-1) to [out=0,in=-90] (0.5,-0.25);
\draw [blue] (0.5,0.25) to [out=90,in=-120] (0.75,1);
\draw [red] (0.75,1) to [out=60,in=-90] (1,1.75);
\draw [red] (0.75,1) to [out=-60,in=90] (1,0.25);
\draw [red] (1,-0.25) to [out=-90,in=120] (1.25,-1) to [out=-120,in=90] (1,-1.75);
\draw [blue] (0.5,-1.75) to [out=90,in=180] (1.25,-1) to [out=60,in=-90] (1.5,-0.25);
\draw [blue] (1.5,0.25) to [out=90,in=180] (2.25,1) to [out=60,in=-90] (2.5,1.75);
\draw [blue] (1.25,-1) to [out=-60,in=90] (1.5,-1.75);
\draw [red] (1.25,-1) to [out=0,in=-90] (2,-0.25);
\draw [red] (2,0.25) to [out=90,in=-120] (2.25,1)to [out=120,in=-90] (2,1.75); 
\draw [red] (2,-1.75) to [out=90,in=180] (2.75,-1) to [out=60,in=-90] (3,-0.25);
\draw [red] (3,0.25) to [out=90,in=180] (3.5,1) -- (3.75,1);
\draw [blue] (2.5,-1.75) to [out=90,in=-120] (2.75,-1) to [out=120,in=-90] (2.5,-0.25);
\draw [blue] (2.5,0.25) to [out=90,in=-60] (2.25,1);
\draw [red] (2.25,1) to [out=0,in=-90] (3,1.75);
\draw [blue] (2.75,-1) -- (3.75,-1);
\draw [red] (2.75,-1) to [out=-60,in=90] (3,-1.75);
\end{tikzpicture}
\]
\end{rmk}

\subsection{Donaldson-Thomas Transformation}

The cluster DT transformation is a central element of the cluster modular group acting on the associated cluster varieties. Combinatorially, a cluster DT transformation can be manifested as a maximal green sequence, or more generally, a reddening sequence of quiver mutations \cite{KelDT}.

\begin{lem}\label{6.7} For any braid word $\beta$, we have  $\Phi_\R=\DT^{-2}$ on $\Aug\left(\Lambda_\beta\right)$.
\end{lem}

\begin{proof} Each positive braid $[\beta]$ defines a double Bott-Samelson cell $\conf^e_\beta(\cC)$ associated with the group $\G=\SL_n$ (see \cite{SWflag} for definition). Theorem 4.10 of \cite{GSW} constructs an algebraic variety isomorphism $\gamma$ between $\Aug\left(\Lambda_\beta\right)$ and $\conf^e_\beta(\cC)$. The cluster $\mathrm{K}_2$ structure on $\Aug\left(\Lambda_\beta\right)$ is inherited from $\conf^e_\beta(\cC)$ via the isomorphism $\gamma$. 

The {left reflection} $_ir$ from $\conf_{s_i\gamma}^\delta(\cC)$ to $\conf_\gamma^{s_i\delta}(\cC)$ and the {right reflection} $r^i$ from $\conf_\gamma^{\delta s_i}(\cC)$ to $\conf_{\gamma s_i}^\delta(\cC)$ are  biregular isomorphisms between double Bott-Samelson cells. By \cite[Corollary 5.4]{GSW}, the isomorphism $\gamma$ intertwines the cyclic rotation $\Phi_\rho$ on augmentation varieties and the isomorphism $r^i\circ {_ir}$ on double Bott-Samelson cells. That is, the following diagram commutes.
\begin{equation}\label{comm diag}
\vcenter{\vbox{\xymatrix@R=1.5em{\Aug\left(\Lambda_{s_i\beta'}\right) \ar@{->}[r]^{\gamma}_\cong \ar[d]_{\Phi_{\rho}}^\cong & \conf^e_{s_i\beta'}(\mathcal{C}) \ar[d]^{r^i\circ {_ir}}_{\cong}\\
\Aug\left(\Lambda_{\beta' s_i}\right) \ar@{->}[r]_{\gamma}^\cong & \conf^e_{\beta' s_i}(\mathcal{C})}}}
\end{equation}

Let $\beta=s_{i_1}\ldots s_{i_l}$. By \cite{SWflag}, the inverse of the DT transformation on  $\conf^e_\beta(\cC)$ is given by
\[
\DT^{-1}=\left(r^{i_l}\circ r^{i_{l-1}} \circ \dots \circ r^{i_1}\right)\circ t,
\]
where $t$ is a biregular isomorphism induced by the transposition action on $\G=\SL_n$. Then
\begin{align*}
\DT^{-2}
=&\left(r^{i_l}\circ r^{i_{l-1}} \circ \dots \circ r^{i_1}\right)\circ t\circ\left(r^{i_l}\circ r^{i_{l-1}} \circ \dots \circ r^{i_1}\right)\circ t\\
=&\left(r^{i_l}\circ r^{i_{l-1}}\circ \cdots \circ r^{i_1}\right)\circ \left({_{i_l}r}\circ {_{i_{l-1}}r} \circ \cdots \circ {_{i_1}r}\right)\circ t\circ t\\
=&\left(r^{i_l}\circ r^{i_{l-1}}\circ \cdots \circ r^{i_1}\right)\circ \left({_{i_l}r}\circ {_{i_{l-1}}r} \circ \cdots \circ {_{i_1}r}\right) \\
=&\left(r^{i_l}\circ {_{i_l}r}\right)\circ \left(r^{i_{l-1}}\circ {_{i_{l-1}}r}\right)\circ \cdots \circ \left(r^{i_1} \circ {_{i_1}r}\right).
\end{align*}
By translating the identity $\DT^{-2}=\left(r^{i_l}\circ {_{i_l}r}\right)\circ  \cdots \circ \left(r^{i_1} \circ {_{i_1}r}\right)$ to the augmentation variety side according to the commutative diagram \eqref{comm diag} we get that
$\Phi_\R=\DT^{-2}$ on  $\Aug\left(\Lambda_\beta\right)$.
\end{proof}

\subsection{Aperiodic DT Yields Infinitely Many Lagrangian Fillings}

\begin{thm}\label{5.11} For any braid word $\beta$, if the $\DT$ transformation on $\Aug\left(\Lambda_\beta\right)$ is aperiodic, then $\Lambda_\beta$ admits infinitely many admissible fillings.
\end{thm}
\begin{proof} Let $L_0$ be the admissible filling that pinches the crossings in $\beta$ from left to right and then fills the resulted unlinks with minimum cobordisms. Let $L_m = \R^{m}\circ L_0$. We claim that $L_m$ is not Hamiltonian isotopic to $L_k$ for $m \neq  k$. To see this, note that by Lemma \ref{6.7}, the cluster seeds of $L_m$ can be computed by mutating the initial seed according to $\DT^{-2m}$; the aperiodicity of $\DT$ implies that the cluster seeds of $L_m$ and $L_k$ are distinct for $m\neq k$. The statement follows from \cite[Theorem 1.3]{GSW}.
\end{proof}

\begin{rmk}

The torus $(n,m)$-link $\Lambda_{(n,m)}$ is the rainbow closure of the $n$-strand braid word $\beta = (s_1s_2 \dotsb s_{n-1})^m$. K\'alm\'an \cite{Kalman} defined a Legendrian loop $\K=\rho^{n-1}$ for $\Lambda_{(n,m)}$. By definition, the full cyclic rotation $\R = \K^{m}.$  The  induced action $\Phi_\K$ and $\Phi_\R$ on the augmentation variety are of finite order.

The quivers associated to $\Aug\left(\Lambda_{(n,m)}\right)$ and those associated to the Grassmannian $\Gr_{n,n+m}$ share the same unfrozen parts. Hence, their DT transformations have the same order. The $\DT$ on $\Gr_{n,n+m}$ has finite order because it is related to the periodic Zamolodchikov operator by $\DT^2=\mathrm{Za}^{m}$ \cite{Kelperiod,weng, SWflag}. In fact, K\'alm\'an's loop induces the Zamolodchikov operator. Summarizing,
$$\Phi_\R = \Phi_\K^m = \DT^{-2}  = \mathrm{Za}^{-m}.$$
\end{rmk}{}

\begin{thm}\label{6.13} Let $Q$ be an acyclic quiver. Its associated $\DT$ transformation is of finite order if and only if $Q$ is of finite type. 
\end{thm}
\begin{proof} Combinatorially, the DT transformation arises from a maximal green sequence of quiver mutations \cite{KelDT}. When $Q$ is acyclic, one may label the vertices of $Q$ by $1,\ldots, l$ such that $i<j$ if there is an arrow from $i$ to $j$. The mutation sequence $\mu_n\circ\cdots \circ \mu_1$ is  maximal green and therefore gives rise to the DT transformation associated with $Q$.

The DT transformation acts the cluster variety $\mathscr{A}_Q$ associated with the quiver $Q$. Following \cite{lee2018frieze}, the {\it frieze variety} $X(Q)$ is defined to be the Zariski closure of the DT-orbit containing the point $P=(1, \ldots, 1)\in \mathscr{A}_Q$. 
Theorem 1.1 of {\it loc.cit.} states that
\begin{enumerate}
    \item If $Q$ is representation finite then the frieze variety $X(Q)$ is of dimension $0$.
    \item If $Q$ is tame then the frieze variety $X(Q)$ is of dimension $1$.
    \item If $Q$ is wild then the frieze variety $X(Q)$ is of dimension at least $2$.
\end{enumerate}{}
As a direct sequence, if $Q$ is not of finite type, then the DT-orbit of $P$ contains infinitely many points, and therefore DT is not periodic. 
If $Q$ is of finite type, then its corresponding cluster variety is of finite type. Hence, its DT transformation is periodic.
\end{proof}

\begin{rmk} Keller pointed out to us that the aperiodicity of $\DT$ for acyclic quiver $Q$ of infinite type follows from the aperiodicity of the Auslander-Reiten translation functor on the derived category of representations of $Q$.
\end{rmk}

\begin{cor}\label{rainbowinfinite}
For any braid word $\beta$, if $Q_\beta$ is acyclic and of infinite type, then $\Lambda_\beta$ admits infinitely many admissible fillings.
\end{cor}
\begin{proof} It follows from Theorem \ref{5.11} and Theorem \ref{6.13}.
\end{proof}

\section{Infinitely Many Fillings for Infinite Type} \label{sec 3}

This section is devoted to the proof of the following result.

\begin{thm}\label{MainStep1}
If $[\beta]$ is a positive braid of infinite type, then the positive braid Legendrian link $\Lambda_\beta$ admits infinitely many non-Hamiltonian isotopic exact Lagrangian fillings.  
\end{thm}

\begin{defn} Given two braid words $\beta$ and $\gamma$, we say $\beta$ \emph{dominates} $\gamma$ if there exists an admissible cobordism from $\Lambda_\gamma$ to $\Lambda_\beta$. Dominance is a partial order on braid words.
\end{defn}

Recall that a quiver is \emph{connected} if its underlying graph is connected. Connectedness of quivers is invariant under  mutations. Under the connectedness assumption, Theorem \ref{MainStep1} is a consequence of Corollary \ref{rainbowinfinite} and the following Propositions.

\begin{prop}\label{3.2} Suppose $\beta$ dominates $\gamma$. If $\Lambda_\gamma$ admits infinitely many  admissible fillings, then so does $\Lambda_\beta$. 
\end{prop}
\begin{proof} Recall from \cite[Theorem 1.4]{GSW} that the induced cluster charts on the augmentation variety can be used to distinguish admissible fillings, and the functorial morphism between augmentation varieties induced by any admissible cobordism maps distinct cluster charts to distinct cluster charts.
\end{proof}

\begin{prop}\label{Mainquiver}
For any braid word $\beta$ with connected $Q_\beta$, either one of the following two scenarios happens:
\begin{enumerate}[label*=\emph{(\arabic*)},leftmargin = *]
    \item there is an admissible concordance from $\Lambda_\gamma$ to $\Lambda_\beta$ and $Q_\gamma$ is a quiver of finite type.
    \item $\beta$ dominates a braid word $\gamma$ and $Q_\gamma$ is acyclic and of infinite type. 
\end{enumerate} 
\end{prop}

\begin{prop}\label{3.4} If Proposition \ref{Mainquiver} (1) happens, then $[\beta]$ is of finite type.

If Proposition \ref{Mainquiver} (2) happens, then $[\beta]$ is of infinite type.
\end{prop}
\begin{proof} Admissible concordances give rise to sequences of  mutations (\cite[\S 5]{GSW}). 
If Proposition \ref{Mainquiver} (1) happens, then $Q_\beta$ is mutation equivalent to $Q_\gamma$. The latter is of finite type. Therefore $[\beta]$ is of finite type. 

If Proposition \ref{Mainquiver} (2) happens, then by \cite[Proposition 5.25(2)]{GSW}, $Q_\beta$ is mutation equivalent to a quiver which contains $Q_\gamma$ as a full subquiver. Suppose that $[\beta]$ is of finite type. Then $Q_\gamma$ is mutation equivalent to finite type quiver, which contradicts with the assumption that $Q_\gamma$ is acyclic and of infinite type. Therefore $[\beta]$ is of infinite type. 
\end{proof}

Proposition \ref{3.4} implies the exclusiveness of the two scenarios of Proposition \ref{Mainquiver}. To conclude the proof of the latter, it remains to prove that the two scenarios in Proposition \ref{Mainquiver} cover all braid words with connected quivers. The strategy of our proof is as follows.
\begin{itemize}
    \item Suppose there is an admissible concordance $\Lambda_\gamma\rightarrow \Lambda_\beta$  such that $Q_\gamma$ is acyclic. If $Q_\gamma$ is of finite type,  then $\beta$ satisfies (1); otherwise, $\beta$ satisfies (2). 
    \item Otherwise, we prove that $\beta$ satisfies (2).
\end{itemize}

\subsection{Preparation}
We  adopt the following notations for  operations on braid words.
\begin{enumerate}
    \item[1.] $\overset{\text{R1}}{=}$ denotes the positive Markov destabilization, which deletes the $s_1$ (resp. $s_{n-1}$) if it only occurs once in $\beta$.
   
    \item[2.] $\overset{\text{R3}}{=}$ denotes the braid move R3, which switches $s_is_{i+1}s_i$ and $s_{i+1}s_is_{i+1}$.
    
    \item[3.] $\overset{\rho}{=}$ denotes the cyclic rotation, which turns $\beta s_i$ into $s_i\beta$ or vice versa.
    
    \item[4.] $\overset{c}{=}$ denotes the commutation which turns $s_is_j$ into $s_js_i$ whenever $|i-j|>1$.

    \item[5.] $\succ$ denotes deleting letters; $\beta\succ \gamma$ means that $\gamma$ can be obtained by deleting letters in $\beta$. In particular, when $\beta\succeq \gamma$, we say that $\gamma$ is a \emph{subword} of $\beta$.
    
    \item[6.] $\overset{\textrm{oppo}}{\rightsquigarrow}$ denotes taking the opposite word ${\beta}^{\mathrm{op}}$. The quiver $Q_{\beta^{\mathrm{op}}}$ alters the orientation of every arrow in $Q_\beta$. 
\end{enumerate}
Operations 1 - 4 induce Legendrian isotopies between corresponding positive braid Legendrian links, which are building blocks for admissible concordance. Operations  5 induces pinch cobordisms between Legendrian links. Operation 6 is a symmetry that can be used to reduce the number of cases considered in the proof.

\begin{lem}\label{basicbraidsinf}
The quivers for the following braids are acyclic and of infinite type:
\begin{enumerate}[label*=\emph{(\arabic*)}]
    \item $s_1^2s_2^2s_1^2s_2^2$, or more generally, $s_i^2s_{i+1}^2s_i^2s_{i+1}^2$;
    \item $s_1s_3s_2^2s_1s_3s_2^2$.
\end{enumerate}
\end{lem}
\begin{proof}
The quivers for (1) and (2) are $\tilde{\mathrm{D}}_5$ and $\tilde{\mathrm{D}}_4$ respectively. 
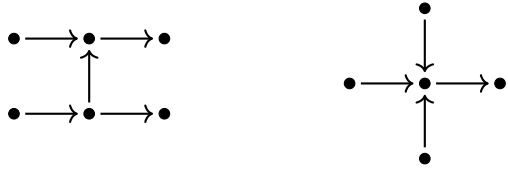
\[
\begin{tikzpicture}
	\draw [thick,->](0.15,0) -- (0.85, 0);
	\filldraw (0,0) circle (2pt);
	\filldraw (1,0) circle (2pt);
	\filldraw (-1,0) circle (2pt);
	\filldraw (0,-1) circle (2pt);
	\filldraw (1,-1) circle (2pt);
	\filldraw (-1,-1) circle (2pt);
    \draw [thick,->] (0.15,-1) -- (0.85, -1);
	\draw [thick,<-](0,-0.15) -- (0, -0.85);
	\draw [thick,<-](-0.15,0) -- (-0.85, 0);
	\draw [thick,->](-0.85,-1) -- (-0.15,-1);
\end{tikzpicture}
\qquad\qquad\qquad
\begin{tikzpicture}[baseline = -15]
	\draw [thick,->](0.15,0) -- (0.85, 0);
	\filldraw (0,0) circle (2pt);
	\filldraw (1,0) circle (2pt);
	\filldraw (-1,0) circle (2pt);
	\filldraw (0,1) circle (2pt);
	\filldraw (0,-1) circle (2pt);
	\draw [thick,<-](0,0.15) -- (0, 0.85);
	\draw [thick,<-](0,-0.15) -- (0, -0.85);
	\draw [thick,<-](-0.15,0) -- (-0.85, 0);
\end{tikzpicture} \qedhere
\]
\end{proof}

\begin{lem}  \label{ws4lemma}
Suppose $w_1,w_2,w_3\in \left\{ s_1s_3, s_1^2, s_3^2\right\}$. Then $w_1s_2w_2s_2w_3s_2w_4s_2$ dominates a braid with an acyclic quiver of infinite type.
\end{lem}
\begin{proof} Note that $\beta = w_1 s_2 {\color{red}{w_2}} s_2 w_3 s_2 {\color{red}{w_4}} s_2 \succ w_1s_2^2 w_3s_2^2.$
If $w_1=w_3$, then the Lemma follows from Lemma \eqref{basicbraidsinf}. The same argument applies to  $w_2=w_4$. 
In the rest of the proof, we assume that $w_1 \neq w_3$ and  $w_2 \neq w_4$.

Let $k$ be the size of the set $\{i \mid w_i = s_1s_3\}$. Here $k\leq 2$; otherwise,  $w_1 = w_3$ or  $w_2 = w_4$. Using the symmetry between $s_1$ and $s_3$, we further assume that there are more $s_1^2$ than $s^2_3$ in $\{w_1, w_2, w_3, w_4\}$. 
We shall exhaust all the possibilities of $k$.

\vskip 2mm

\paragraph{{\it Case 1: k=2}} After taking necessary cyclic rotations and/or the opposite word, we have $w_1 = w_2 = s_1s_3$, and the values of $w_3, w_4$ split into two subcases.

If $w_3 =s_1^2$ and $w_4 = s_3^2$, then $\Lambda_\beta$ is admissible concordance to the standard $\mathrm{E}_9$ link:
                \begin{align*}
                    \beta
                    &= s_1s_3s_2s_3s_1s_2s_1s_1 {\color{teal}{s_2s_3s_3s_2}}
                    \stackrel{\rho}{=}
                    s_2s_3s_3s_2s_1 {\color{blue}{s_3s_2s_3}} s_1s_2s_1s_1 \\
                    &\stackrel{\textrm{R3}}{=}
                    s_2s_3s_3 {\color{blue}{s_2s_1s_2}} s_3s_2 s_1s_2s_1s_1
                    \stackrel{\textrm{R3}}{=}
                    s_2{\color{blue}{s_3s_3s_1}}s_2 {\color{blue}{s_1s_3}}s_2 s_1s_2s_1s_1 \\
                    &
                    \stackrel{c}{=}
                    s_2s_1{\color{blue}{s_3s_3s_2s_3}}s_1s_2 s_1s_2s_1s_1 
                    \stackrel{\textrm{R3}}{=}
                    s_2s_1s_2{\color{teal}{s_3}}s_2s_2s_1s_2 s_1s_2s_1s_1    \\
                    &\stackrel{\textrm{R1}}{=} {\color{blue}{s_2s_1s_2}}s_2s_2s_1{\color{blue}{s_2 s_1s_2}}s_1s_1 
                    \stackrel{\textrm{R3}}{=}
                    s_1{\color{blue}{s_2s_1s_2s_2}}s_1s_1 s_2s_1s_1s_1    \\
                    &\stackrel{\textrm{R3}}{=}
                    s_1s_1s_1s_2s_1s_1s_1 s_2s_1s_1s_1 
                    =
                    s_1^3 s_2s_1^3 s_2 {\color{teal}{s_1^3}} 
                    \stackrel{\rho}{=} s_1^6 s_2s_1^3 s_2
                \end{align*}
 
 If ${w_3 =w_4 =s_1^2}$, then 
            \begin{align*}
                \beta 
                &= s_1 {\color{blue}{s_3 s_2 s_3}} s_1 s_2 s_1^2 s_2 s_1^2 s_2 
                \stackrel{\textrm{R3}}{=}
                s_1s_2 {\color{teal}{s_3}} s_2 s_1s_2 s_1^2 s_2 s_1^2 s_2 
                \stackrel{\textrm{R1}}{=}
                s_1s_2^2 s_1s_2 s_1^2 s_2 s_1^2 s_2 \\
                &= s_1s_2^2 s_1s_2 s_1 s_1 s_2 s_1 {\color{teal}{s_1 s_2}} 
                \stackrel{\rho}{=} 
                s_1 s_2 s_1s_2^2 s_1 s_2 s_1 {\color{blue}{s_1 s_2 s_1}}  
                \stackrel{\textrm{R3}}{=} 
                s_1 {\color{red}{s_2}} s_1s_2^2 s_1 {\color{red}{s_2}} s_1 s_2 {\color{red}{s_1}} s_2
                {\succ}
                s_1^2s_2^2s_1^2s_2^2.
            \end{align*}
            
    \vskip 2mm
\paragraph{{\it Case 2: k=1}} We assume that $w_1 =s_1s_3$ after a necessary cyclic rotation. Then $w_2,w_3,w_4$ are either $s_1^2$ or $s_3^2$. Note that $w_2\neq w_4$.  By the symmetry between $s_1$ and $s_3$, and taking rotations and the opposite word if necessary, it suffices to consider $w_2=w_3=s_1^2$ and $w_4=s_3^2$. The $\Lambda_\beta$ is admissible concordance to the standard $\mathrm{E}_9$ link:
        \begin{align*}
            \beta 
            &=
            s_3s_1s_2s_1^2s_2s_1^2{\color{teal}{s_2s_3^2s_2}}
            \stackrel{\rho}{=}
            s_2{\color{blue}{s_3^2s_2s_3}}s_1s_2s_1^2s_2s_1^2
            \stackrel{\textrm{R3}}{=}
            s_2s_2^2{\color{teal}{s_3}}s_2s_1s_2s_1^2s_2s_1^2 \\
            &\stackrel{\textrm{R1}}{=}
            s_2s_2^2s_2s_1s_2s_1^2s_2s_1^2
            =
            {\color{blue}{s_2^4s_1s_2}}s_1^2s_2s_1^2
            \stackrel{\textrm{R3}}{=}
            s_1^4s_2s_1s_1^2s_2{\color{teal}{s_1^2}}
            \stackrel{\rho}{=}
            s_1^6s_2s_3s_2.
        \end{align*}
  \vskip 2mm      
 \paragraph{{\it Case 3: k=0}} Assume that $w_1=w_2 =s_1^2$ and $w_3=w_4 = s_3^2$. Then $Q_\beta$ is of type $\tilde{\mathrm{D}}_8$:
    
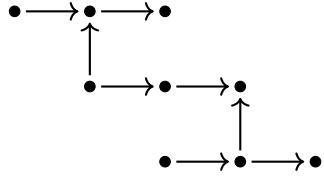
\[
    \begin{tikzpicture}
	\draw [thick,->](0.15,0) -- (0.85, 0);
	\filldraw (0,0) circle (2pt);
	\filldraw (1,0) circle (2pt);
	\filldraw (-1,0) circle (2pt);
	\filldraw (0,-1) circle (2pt);
	\filldraw (1,-1) circle (2pt);
	\filldraw (2,-1) circle (2pt);
		\filldraw (3,-2) circle (2pt);
	\filldraw (2,-2) circle (2pt);
	\filldraw (1,-2) circle (2pt);
    \draw [thick,->] (0.15,-1) -- (0.85, -1);
	\draw [thick,<-](0,-0.15) -- (0, -0.85);
	\draw [thick,<-](-0.15,0) -- (-0.85, 0);
	\draw [thick,->](1.15,-1) -- (1.85,-1);
	\draw [thick,->](1.15,-2) -- (1.85,-2);
	\draw [thick,->](2.15,-2) -- (2.85,-2);
	\draw [thick,<-](2,-1.15) -- (2, -1.85);
	\end{tikzpicture} \qedhere
    \]
\end{proof}

\begin{defn}\label{defn:subword}
Let $\beta$ be a braid word of $n$ strands. For $1\leq i < j\leq n-1$, we define
$$\beta(i,j) := \textrm{the sub-word of }{\beta} \textrm{ that contains } s_i, s_{i+1},\dotsb, s_j.$$
For example, if $\beta = s_1 s_2s_3 s_1^2 s_2 s_5s_2s_3s_4$, then $\beta(2,3) = s_2s_3s_2^2s_3$.
\end{defn}

\begin{lem}\label{quiverdeg} Let $\beta$ be a braid word of $n$ strands.
\begin{enumerate}[label=\emph{(\arabic*)}]
    \item If $s_i^2$ is not a subword of $\beta$, then $Q_\beta=Q_{\beta(1,i-1)}\sqcup Q_{\beta(i+1,n-1)}$.

    \item If $\beta(i,i+1)$ does not contain a sub-word of intertwining pairs, namely neither $s_{i}s_{i+1}s_{i}s_{i+1}$ nor $s_{i+1}s_{i}s_{i+1}s_{i}$, then $Q_\beta=Q_{\beta(1,i)}\sqcup Q_{\beta(i+1,n-1)}$.
\end{enumerate}
\end{lem}
\begin{proof} The brick diagram has an empty level $i$ in case (1) and does not have arrows between level $i$ and level $i+1$ in case (2).
\end{proof}

\begin{lem}\label{3.10}  Let $n\geq 3$ and let $\beta$ be an $n$-strand braid word such that $Q_\beta$ is connected.  If $\beta\succ s_1^2$ and  $\beta \succ s_{n-1}^2$, then $Q_\beta$ is acyclic if and only if for $1\leq i\leq n-2$, we have \[\beta(i,i+1)=s_i^{a_1}s_{i+1}^{b_1}s_i^{a_2}s_{i+1}^{b_2} ~\mbox{or}~ s_{i+1}^{b_1}s_i^{a_1}s_{i+1}^{b_2}s_i^{a_2}, \quad \quad \mbox{where } a_1, a_2, b_1, b_2\geq 1\]
\end{lem}
\begin{proof} The \emph{if} direction is obvious. To see the \emph{only if} direction, let us assume without loss of generality that $\beta(i,i+1)$ begins with $s_i$. If $\beta(i,i+1)$ does not end after $s_i^{a_1}s_{i+1}^{b_1}s_i^{a_2}s_{i+1}^{b_2}$, then there is at least one $s_i$ after $s_{i+1}^{b_2}$, giving $Q_\beta$ an $a_2$-cycle between levels $i$ and $i+1$.
\end{proof}

\begin{assumption}\label{mainassumption}
Note that $2$-strand braids correspond to type A quivers. It suffices to consider  braid words $\beta$ of at least 3 strands. Let us single out the generator $s_2$. After necessary rotations, we assume that $\beta$ does not start with $s_2$ but ends with $s_2$, that is, 
$$\beta = w_1s_2^{b_1} w_2 s_2^{b_2}\dotsb w_m s_2^{b_m},$$
where each $w_i$ is a word of $s_1, s_3, s_4, \dotsb, s_{n-1}$. 

We assume that every $w_i$ contains at least one $s_1$ or $s_3$; otherwise, we can move the whole $w_i$ across the $s_2$'s at either end and merge it with $w_{i-1}$ or $w_{i+1}$. We further assume that $\sum b_i$ achieves minimum. Under this assumption, the length of every $w_i(1,3)$ is at least  $2$. Otherwise, with the letters $s_4\dotsb, s_{n-1}$ migrated away, we have $s_2w_is_2=s_2s_1s_2$ or $s_2s_3s_2$, and we can use R3 to reduce $\sum b_i$. 

We assume that $m\geq 2$; otherwise, $Q_\beta$ is disconnected by Lemma \ref{quiverdeg}. Meanwhile, if $m\geq 4$, then after necessarily deleting letters, we land on the case of Lemma \ref{ws4lemma}, and the braid $\beta$ dominates a braid with an acyclic quiver of infinite type.

In the rest of this section, without loss of generality, we assume that 
\begin{equation}
\label{beta.exp}
\beta = w_1s_2^{b_1}w_2s_2^{b_2}\dotsb w_ms_2^{b_m},
\end{equation}where  $b_i \geq 1$, $m=2$ or $3$, and $w_i\succeq s_1^2, s_3^2$ or $s_1s_3$.
\end{assumption}
We prove Proposition \ref{Mainquiver} by induction on the number of strands of $\beta$.

\subsection{Proof of Proposition \ref{Mainquiver} for 3-strand braids{}} \label{sec.3.2}
 If $m=2$ in \eqref{beta.exp}, then $Q_\beta$ is acyclic and therefore the proposition follows. It remains to consider $m=3$.  Suppose that at least one of the $b_i$'s, say $b_3$ after necessary cyclic rotations, is greater than 1.  The proposition follows since 
  \[\beta \succeq w_1s_2{\color{red}{w_2}}s_2w_3s_2^2
    {\succ} 
    w_1s_2^2w_3s_2^2\succeq s_1^2s_2^2s_1^2s_2^2.\]
It remains to consider $b_1=b_2=b_3=1$, i.e., 
\[\beta = s_1^{a_1}s_2s_1^{a_2}s_2s_1^{a_3}s_2.\] 

If two of $a_i$'s, say $a_1$ and $a_2$ after necessary rotations, are equal to 2, then
        \begin{align*}
            \beta &= s_1{\color{blue}{s_1s_2s_1}}s_1s_2s_1^{a_3}s_2 
                \stackrel{\text{R3}}{=} s_1s_2s_1{\color{blue}{s_2s_1s_2}}s_1^{a_3}s_2 
                \stackrel{\text{R3}}{=} s_1s_2s_1s_1s_2s_1s_1^{a_3}{\color{teal}{s_2}} \\
                &\stackrel{\rho}{=} {\color{blue}{s_2s_1s_2}}s_1s_1s_2s_1s_1^{a_3} 
                \stackrel{{\text{R3}}}{=} {\color{teal}{s_1}}s_2s_1s_1s_1s_2s_1s_1^{a_3}
                \stackrel{{\rho}}{=} s_2s_1s_1s_1s_2s_1s_1^{a_3}s_1 = s_2s_1^3s_2s_1^{a_3+1}.
        \end{align*}
The quiver for the last word is  acyclic. The proposition is proved.

Otherwise, at least two of the $a_i$'s, say $a_1$ and $a_2$ after necessary rotations, are greater than 2. The proposition follows since 
        \begin{align*}
        \beta 
            &\succ s_1^{3}s_2s_1^{3}s_2s_1^{2}s_2 
            = s_1^{2}{\color{blue}{s_1s_2s_1}}s_1^{2}s_2s_1^{2}s_2 
            \stackrel{\text{R3}}{=} 
                s_1^{2}s_2{\color{red}{s_1}}s_2s_1^{2}s_2{\color{red}{s_1^{2}}}s_2
            {\succ} s_1^2s_2^2s_1^2s_2^2.
        \end{align*}

\subsection{Proof of Proposition \ref{Mainquiver} for braids of at least 4 strands{}}
Assume that $\beta$ is expressed as in \eqref{beta.exp}. Note that $s_1$ commutes with all other generators in $w_i$. Therefore we further assume that 
\begin{itemize}
    \item $w_i = s_1^{a_i}v_i = v_i s_1^{a_i}$, where $v_i$ is a word of $s_3, \dotsb, s_{n-1}$.
\end{itemize}
We shall start with the proof of the following two lemmas. 

\begin{lem}\label{3.13} Suppose $\beta \succ s_1$ and $\beta \succ s_2$. If $Q_{\beta(1,2)}$ is of Dynkin type $\mathrm{A}$, then there exists an admissible concordance $\Lambda_\gamma\rightarrow \Lambda_\beta$ such that $\gamma$ has fewer strands than $\beta$.
\end{lem}
\begin{proof} Since $Q_{\beta(1,2)}$ is of Dynkin type $\mathrm{A}$, $\beta(1,2)$ must be of the form $s_1^{a_1}s_2^{b_1}s_1^{a_2}s_2^{b_2}$ with $\min\{a_1, a_2\}=\min\{b_1,b_2\}=1$. After necessary cyclic rotations and/or taking the opposite word, we assume $a_1=b_1=1$. Then
\[
\beta = v_1 {\color{blue}s_1 s_2 s_1^{a_2}} v_2 s_2^{b_2} \stackrel{\textrm{R3}}{=} v_1 s_2^{a_2} {\color{teal}s_1} s_2 v_2 s_2^{b_2}   \stackrel{\textrm{R1}}{=}v_1 s_2^{a_2} s_1 s_2 v_2 s_2^{b_2}.
\]
The braid reduces to the case of one less strand. 
\end{proof}

\begin{lem} \label{A sub-tree} Suppose $Q_\beta$ is connected. If $Q_{\beta(1,3)}$ is acyclic and $Q_{\beta(1,2)}$ is not of type $\mathrm{A}$, then Proposition \ref{Mainquiver} is true for $[\beta]$.
\end{lem}
\begin{proof} Define 
\[k: = \max\{i ~|~  Q_{\beta(1,i)} \mbox{ is acyclic}\}.
\]
If $k=n$, then $Q_\beta$ is acyclic and the Lemma is proved. If $Q_{\beta(1,k)}$ is of infinite type, then the Lemma follows since  $\beta \succ \beta(1,k)$. Now we assume $k<n$ and $Q_{\beta(1,k)}$ is of finite type.

Note that $Q_{\beta(1,3)}$ is a subquiver of $Q_{\beta(1,k)}$. By assumption, $Q_{\beta(1,2)}$ is not of type $\mathrm{A}$. Therefore $Q_{\beta(1,k)}$ must be of type $\mathrm{D}$ or $\mathrm{E}$. Hence, $Q_{\beta(i,j)}$ is of type $\mathrm{A}$ for $1<i<j\leq k$.
In particular, $Q_{\beta(k-1,k)}$ is of type $\mathrm{A}$ and $Q_\beta$ is connected. Therefore we have
$$\beta(k-1,k) = s_{k-1}^{e_1} s_{k}^{f_1} s_{k-1}^{e_2} s_{k}^{f_2},\quad \text{or} \quad \beta(k-1,k) = s_{k}^{f_1} s_{k-1}^{e_1} s_{k}^{f_2}s_{k-1}^{e_2},$$
where
$$\min\{e_1,e_2\} = \min\{f_1,f_2\}=1.$$

Below we consider the first case $\beta(k-1,k) = s_{k-1}^{e_1} s_{k}^{f_1} s_{k-1}^{e_2} s_{k}^{f_2}$. The second case follows by taking the opposite word of $\beta$. The letters $s_1,\dotsb, s_{k-2}$ commute with  $s_{k+1},\dotsb, s_n$. After necessary communications of the letters in $\beta$, we can write
$$\beta = \gamma_1\delta_1\gamma_2\delta_2,$$
where $\gamma_i~ (i=1,2)$ is a word of  $s_1\dotsb, s_{k-1}$ with $e_i$ many of $s_{k-1}$, and $\delta_i$ is a word of $s_k,\dotsb, s_n$ with $f_i$ many of $s_{k}$. We remark that we have \emph{not} performed cyclic rotations yet and will only do it carefully, so that the quiver for $\beta(1,k-1)=\gamma_1\gamma_2$ is not distorted. 

Recall that $\min\{f_1,f_2\} = 1$. We consider the case $f_1=1$. The argument for $f_2=1$ is a similar repetition. Let us write 
$\delta_1 = xs_k y,$
where $x,y$ are words of $s_{k+1}\dotsb, s_n$ and they commute with $\gamma_1,\gamma_2$. We pass $y$ through $\gamma_2$, and we pass $x$ through $\gamma_1$ and rotation, obtaining $\gamma_1(xs_ky)\gamma_2\delta_2 \rightsquigarrow \gamma_1 s_k \gamma_2 (y\delta_2x)$. This move does not change the quiver for $\beta (1,k)$, and is a Legendrian isotopy. Consequently, we can assume $\delta_1 = s_k$  and write $\beta=\gamma_1s_k\gamma_2\delta_2$.

Now we consider 
\[\delta_2(k,k+1)=s_k^{g_1}s_{k+1}^{h_1}s_k^{g_2}s_{k+1}^{h_2}\cdots s_k^{g_l},\] 
where $g_1, g_l\geq 0$ and all other powers $\geq 1$.
The Lemma holds for the following two cases. 
\begin{enumerate}
    \item If $\delta_2\succ s_{k+1}s_k^2s_{k+1}$, then $\beta \succ \gamma_1 s_k \gamma_2s_{k+1}s_k^2s_{k+1}:=\beta_1$. 
    
    \item If $\delta_2 \succ s_{k+1}^2s_ks_{k+1}^2$, then $\beta \succ \gamma_1 s_k \gamma_2s_{k+1}^2s_ks_{k+1}^2:=\beta_2$. 
\end{enumerate}
The quivers for $\beta_1$ and $\beta_2$ are acyclic and of infinite type, as depicted below: 
$$
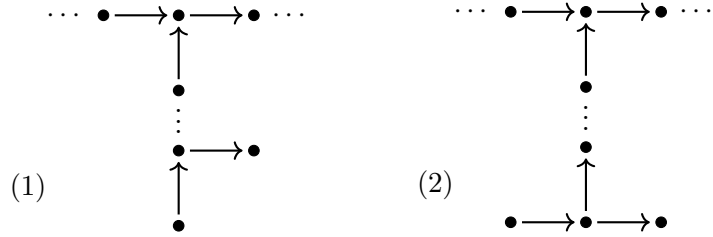
\begin{tikzpicture}
    \filldraw (0,-1) circle (2pt);
    \filldraw (-1,0) circle (2pt);
    \node at (-1.5,0) [] {$\dotsb$};
    \filldraw (0,0) circle (2pt);
    \filldraw (1,0) circle (2pt);
    \node at (1.5,0) [] {$\dotsb$};
    \draw [thick,->](0.15,0) -- (0.85, 0);
    \draw [thick,<-](-0.15,0) -- (-0.85, 0);
    \draw [thick,<-](0,-0.15) -- (0, -0.85);
    \node at (0, -1.3) [] {\vdots};
    \filldraw (0,-1.8) circle (2pt);
    \filldraw (1,-1.8) circle (2pt);
    \filldraw (0,-2.8) circle (2pt);
    \draw [thick,->](0.15,-1.8) -- (0.85, -1.8);
    \draw [thick,->](0,-2.65) -- (0, -1.95);
    \node at (-2, -2.3) [] {$(1)$};
\end{tikzpicture}
\qquad\quad
\begin{tikzpicture}
    \filldraw (0,-1) circle (2pt);
    \filldraw (-1,0) circle (2pt);
    \node at (-1.5,0) [] {$\dotsb$};
    \filldraw (0,0) circle (2pt);
    \filldraw (1,0) circle (2pt);
    \node at (1.5,0) [] {$\dotsb$};
    \draw [thick,->](0.15,0) -- (0.85, 0);
    \draw [thick,<-](-0.15,0) -- (-0.85, 0);
    \draw [thick,<-](0,-0.15) -- (0, -0.85);
    \node at (0, -1.3) [] {\vdots};
    \filldraw (0,-1.8) circle (2pt);
    \filldraw (1,-2.8) circle (2pt);
    \filldraw (-1,-2.8) circle (2pt);
    \filldraw (0,-2.8) circle (2pt);
    \draw [thick,->](0.15,-2.8) -- (0.85, -2.8);
    \draw [thick,->](-0.85,-2.8) -- (-0.15, -2.8);
    \draw [thick,->](0,-2.65) -- (0, -1.95);
    \node at (-2, -2.3) [] {$(2)$};
\end{tikzpicture}
$$

By the definition of $k$, the quiver for $\beta(k,k+1)$ is not acyclic. Therefore we have $\delta_2\succeq s_{k+1}s_ks_{k+1}s_k$. We assume that $\delta_2$ does not satisfy the above $(1)$ or $(2)$. Then
 \[\delta_2(k,k+1)=s_k^{g_1}s_{k+1}^{h_1}s_ks_{k+1}^{h_2}s_k^{g_3},\]  where $g_1\geq 0$, $g_3\geq 1$, and $\min\{h_1,h_2\} =1$. Depending on whether $h_1=1$ or $h_2=1$, we have the following two cases: 
\[
\beta(1,k+1)
=
\gamma_1s_k\gamma_2{\color{blue}s_k^{g_1}s_{k+1}s_k}s_{k+1}^{h_2}s_k^{g_3}
\overset{\text{R3}}{=}
\gamma_1s_k\gamma_2s_{k+1}{\color{teal}s_ks_{k+1}^{g_1+h_2}s_k^{g_3}}
\overset{\rho}{=}
s_ks_{k+1}^{g_1+h_2}s_k^{g_3}\gamma_1s_k\gamma_2s_{k+1},
\]
\[
\beta(1,k+1)
=
\gamma_1s_k\gamma_2s_k^{g_1}s_{k+1}^{h_1}{\color{blue}s_ks_{k+1}s_k^{g_3}}
\overset{\text{R3}}{=}
\gamma_1s_k\gamma_2s_k^{g_1}s_{k+1}^{h_1+g_3}s_ks_{k+1}.
\]
In both cases, the only R3 move is $s_ks_{k+1}s_k=s_{k+1}s_ks_{k+1}$ performed in $\delta_2$, hence the move can be extended from $\beta(1,k+1)$ to $\beta$. The cyclic rotations can also be extended to $\beta$ without changing the quiver for $\beta(1,k)$. In the end, we performed a Legendrian isotopy and get a new braid word $\beta'$ with acyclic $Q_{\beta'(1,k+1)}$. We repeat the above argument for $\beta'$ and $k+1$.
This completes the case $f_1=1$. 
\end{proof}

Now we prove the proposition. If $m=2$ in \eqref{beta.exp},  then $Q_{\beta(1,3)}$ is acyclic. If $Q_{\beta(1,2)}$ is not of type $\mathrm{A}$, then the proposition follows directly from Lemma \ref{A sub-tree}. Otherwise, we apply Lemma \ref{3.13}. It remains to consider the case $m=3$, in which we have
\begin{equation}\beta =w_1s_2^{b_1}w_2s_2^{b_2}w_3s_2^{b_3}= v_1s_1^{a_1}s_2^{b_1}v_2s_1^{a_2}s_2^{b_2}v_3s_1^{a_3}s_2^{b_3}.
\end{equation}

Let us set
\[
p= \# \{i~|~a_i\neq 0 \}, \hskip 14mm q=\#\{i~|~ v_i \succeq s_3\}.
\] Here $p,q \in \{2, 3\}$. We consider cases by $(p,q)$.

\medskip
\paragraph{\it Case 1: $(p,q) = (2,2)$} 
After suitable cyclic rotation, we assume $a_3=0$. Then $Q_{\beta(1,3)}$ is acyclic. The rest goes through the same line as the above proof for the case $m=2$.

 \medskip   
 \paragraph{\it Case 2: $(p,q) = (2,3)$} After suitable cyclic rotation, we assume $a_3=0$. If $b_1 \geq 2$, then 
 \[\beta= v_1s_1^{a_1}s_2^{b_1}v_2s_1^{a_2}s_2^{b_2}v_3s_2^{b_3} \succ v_1s_1s_2^{b_1}v_2s_1s_2^{b_2+b_3} \succeq s_3s_1s_2^2 s_3s_1s_2^2.
 \]The proposition follows. So we assume $b_2=1$. 
 
 Now if $a_2 =1$, then using  $s_1^{a_1}s_2s_1 = s_2s_1s_2^{a_1}$, we can reduce the number of strands. The same argument works for $a_1=1$.
It remain to consider $a_1 \geq 2$ and $a_2\geq 2$. Then 
    $$\beta \succeq 
    s_1^2 {\color{red}{v_1}} s_2 {\color{red}{v_2}} s_1^2 s_2 {\color{red}{v_3}} s_2 
    {\succ} 
    s_1^2 {\color{blue}s_3 s_2 s_3} s_1^2s_2 s_2 
    \stackrel{\textrm{R3}}{=}
    s_1^2 s_2 {\color{red}{s_3}} s_2 s_1^2s_2 s_2
    {\succ} 
    s_1^2 s_2^2  s_1^2s_2^2.
    $$
 
 \medskip   
 \paragraph{\it Case 3: $(p,q) = (3,3)$} We have 
 \[w_i=v_is_i^{a_i}\succeq s_3s_1,  \quad \quad \forall i=1,2,3.\] 
 If there is a $b_i$, say $b_1$ after necessary cyclic rotation, greater than $1$, then 
 \[\beta \succ s_3s_1 s_2^{b_1}s_3s_1s_2^{b_2+b_3}\succeq s_1s_3s_2^2s_1s_3s_2^2.
 \]
The proposition follows. It remains to consider  $b_1= b_2=b_3 =1$.

If there is some $w_i$ with $w_i(1,3) = s_1s_3$, after suitable cyclic rotation we can assume  $w_2(1,3) = s_1s_3$. Note that the rest letters of $w_2$ are $s_4, \ldots, s_n$. They commute with the $s_2$ at either end and can be merged into $w_1$ or $w_3$. Therefore, we may assume $w_2=s_1s_3$ and  use  identity 
   \begin{equation}\label{3strandstocancel1}
        s_1^{a_1}s_2s_1s_3s_2s_1^{a_3} = s_3^{a_3}s_2s_1s_3s_2s_3^{a_1}
    \end{equation}
to reduce the number of strands.
    
    If none of the $w_i$'s has $w_i(1,3) = s_1s_3$. Then $w_i(1,3) \succeq s_1^2s_3$ or $w_i(1,3) \succeq s_1s_3^3$ for $i=1,2,3$. Two of them must be the same kind, and they have adjacent indices after cyclic rotation. For example, if $w_1,w_2$ are of the same type $s_1s_1s_3$, then
    $$\beta \succeq 
    s_1s_3s_1 s_2 s_1s_3s_1 s_2 {\color{red}{w_3}} s_2
    {\succ} 
    s_1s_3 {\color{blue}{s_1s_2s_1}} s_3s_1s_2s_2
    \stackrel{\textrm{R3}}{=}
    s_1s_3 s_2 {\color{red}{s_1}} s_2 s_3s_1s_2s_2
    {\succ} 
    s_1s_3 s_2^2 s_3s_1s_2^2.
    $$
    Other combinations of $w_i$ are similar.

\medskip
 \paragraph{\it Case 4: $(p,q) = (3,2)$} If $\beta$ is a $4$-strand word, then by the symmetry between $s_1$ and $s_3$, it reduces to the case $(p,q)=(2,3)$. Below we assume $\beta$ is at least of $5$ strands.

 After cyclic rotations, we assume that $v_3$ does not contain  $s_3$. Then $v_3$ commutes with $s_2$ and can be merged into $w_2$. Hence we assume that  \[w_1\succ s_1s_3, \quad \quad w_2\succ s_1s_3, \quad \quad w_3 = s_1^{a_3} \mbox{ with } a_3\geq 2.\] 
If $b_1\geq 2$, then the proposition follows since
\[\beta\succeq s_1s_3s_2^2s_1s_3s_2 {\color{red}{w_3}} s_2 
{\succ} 
s_1s_3s_2^2s_1s_3s_2^2.\]  Below we consider $b_1=1$.

If $w_1\succ s_1s_3^2$ and $w_2\succ s_1s_3^2$, then the proposition follows since
\[\beta\succeq 
s_1s_3 {\color{blue}{s_3s_2s_3}} s_3s_1s_2^2
\stackrel{\textrm{R3}}{\succ}
s_1s_3s_2{\color{red}{s_3}}s_2s_3s_1s_2^2
{\succ} 
s_1s_3s_2^2s_1s_3s_2^2.\] 
 Hence, we assume that one of $w_1, w_2$ contains a single $s_3$. After suitable cyclic rotations and taking the opposite word if necessary, we assume that $w_2$ contains a single $s_3$. Moreover, all the letters $s_4,\dotsb, s_{n-1}$ in $w_2$ can be merged to $w_1$ by moving them in two directions and taking necessary cyclic rotations.
To summarize, it remains to consider 
\[\beta = v_1 s_1^{a_1} s_2 s_1^{a_2} s_3 s_2^{b_2} s_1^{a_3} s_2^{b_3},\quad \mbox{where }v_1\succeq s_3,~a_3\geq 2,~ \mbox{and }a_1, a_2, b_2, b_3\geq 1.\]
 We split our proof into two cases based on the value of $a_2$.

\medskip

A. If $a_2 =1$, then $b_2\geq 2$. Otherwise, $\beta =v_1 {\color{purple}{s_1^{a_1} s_2 s_1 s_3 s_2 s_1^{a_3}}} s_2^{b_3}$, and we can apply Identity \eqref{3strandstocancel1} to the purple part to reduce the number of strands. We further assume $b_3 =1$; otherwise, $b_3\geq 2$, and together with $a_3,b_2\geq 2$, we have
    $$\beta \succeq s_1^{a_1}{\color{red}{s_2}}s_1^{a_2}s_2^{b_2}s_1^{a_3}s_2^{b_3}
    {\succ} 
    s_1^{a_1+a_2}s_2^{b_2}s_1^{a_3}s_2^{b_3}
    {\succeq}
    s_1^{2}s_2^{2}s_1^{2}s_2^{2}.
    $$
    To recollect, we have
    \[\beta = v_1 s_1^{a_1} s_2 s_1^{a_2} s_3 s_2^{b_2\geq 2} s_1^{a_3\geq 2} s_2.\] 
    If $w_1(1,3)=s_1s_3$, then $w_1 =x s_1s_3 y$. After rotating $s_1^{a_3} s_2$, and moving $x,y$, we have
    $$
    \beta = 
    x s_1s_3 y s_2 s_1^{a_2} s_3 s_2^{b_2} {\color{teal}{s_1^{a_3} s_2}}
    \stackrel{\rho}{=}
    s_1^{a_3} s_2 {\color{teal}{x}} s_1s_3 {\color{teal}{y}} s_2 s_1^{a_2} s_3 s_2^{b_2}
    \stackrel{\textrm{c},\rho}{=}
    {\color{purple}{s_1^{a_3} s_2 s_1s_3 s_2 s_1^{a_2}}} ys_3 s_2^{b_2}x.
    $$
    We apply identity \eqref{3strandstocancel1} to the purple part to reduce the number of strands. 
    Therefore we can assume $w_1(1,3)\succeq s_1s_3^2$ or $s_1^2s_3$. 
    
    Now we focus on $w_1(1,4)$. The connectedness of $Q_\beta$ implies that $w_1(1,4)$ has at least two copies of $s_4$, with at least one $s_3$ sandwiched in between. Hence there are four possibilities:
    $$
    w_1(1,4)\succeq \; 
    (a)\; s_1^2s_4s_3s_4, \;
    (b)\; s_1s_4s_3s_3s_4, \;
    (c)\; s_1s_4s_3s_4s_3, \;
    (d)\; s_1s_3s_4s_3s_4. 
    $$
The proposition follows via direct calculations:
    \begin{enumerate}
        \item[($a$)] $w_1(1,4)\succeq s_1^2s_4s_3s_4 =  s_1^2s_3s_4s_3\succ s_1^2s_3^2$. Then
        $$
        \quad \quad \quad \beta \succ s_1^2{\color{blue}{s_3^2s_2s_3}}s_1s_2^2s_1^2s_2
        \stackrel{\textrm{R3}^2}{=}
        s_1^2s_2s_3s_2{\color{blue}{s_2s_1s_2}}s_2s_1^2s_2
        \stackrel{\textrm{R3}}{=}
        s_1^2s_2{\color{red}{s_3}}s_2s_1{\color{red}{s_2}}s_1s_2{\color{red}{s_1^2}}s_2
        {\succ} 
        s_1^2s_2^2s_1^2s_2^2.
        $$
        
        \item[($b$)] $w_1(1,4)\succeq s_1s_4s_3s_3s_4$. Then
        \begin{align*}
        \beta 
        & \succeq s_1{\color{teal}{s_4}}s_3^2{\color{teal}{s_4}}s_2s_1s_3s_2^2s_1^2s_2
        \stackrel{\textrm{c},\rho}{=}
        s_1s_3^2s_2s_1{\color{blue}{s_4s_3s_4}}s_2^2s_1^2s_2
        \stackrel{\textrm{R3}}{=}
        s_1s_3^2s_2s_1s_3{\color{red}{s_4}}s_3s_2^2s_1^2s_2 \\
        &{\succ} 
        s_1s_3^2s_2s_1{\color{teal}{s_3}}s_3s_2^2s_1^2s_2
        \stackrel{\textrm{c}}{=}
        s_1s_3{\color{blue}{s_3s_2s_3}}s_1s_3s_2^2s_1^2s_2
        \stackrel{\textrm{R3}}{=}
        s_1s_3s_2{\color{red}{s_3}}s_2s_1s_3s_2^2{\color{red}{s_1^2s_2}}\\
        &{\succ}  s_1s_3s_2^2s_1s_3s_2^2.
        \end{align*}
        
        \item[($c$)]  $w_1(1,4)\succeq s_1s_4s_3s_4s_3 = s_1s_3s_4s_3s_3 \succ s_1s_3^3$. Then
        \begin{align*}
        \beta 
        & \succeq s_1{\color{blue}{s_3^3s_2s_3}}s_1s_2^2s_1^2s_2
        \stackrel{\textrm{R3}}{=}
        s_1s_2{\color{teal}{s_3}}s_2^3s_1s_2^2s_1^2s_2
        \stackrel{\textrm{R1}}{\succ}
        s_1s_2^4s_1s_2^2{\color{teal}{s_1^2s_2}} \\
        &\stackrel{\rho}{=}
        {\color{blue}{s_1^2}}s_2s_1s_2^4s_1s_2^2
        \stackrel{\textrm{R3}}{=}
        {\color{teal}{s_2}}s_1s_2^2s_2^4s_1s_2^2
        \stackrel{\rho}{=}
        s_1s_2^6s_1s_2^3.
        \end{align*}
        We end up with an $\mathrm{E}_9$ quiver, which is acyclic and of infinite type.
        
        \item[($d$)] $w_1(1,4)\succeq s_1s_3s_4s_3s_4 = s_1s_3s_3s_4s_3 \succ s_1s_3^3$. The rest follows from the same calculation as in $(c)$.
    \end{enumerate}

\medskip    

B. If $a_2\geq 2$, then we look at $w_1(1,3)$. 

If $w_1(1,3) = s_1s_3$, then $v_1=xs_3y$, where $x,y$ are words of $s_4, \ldots, s_{n-1}$.  Let $\tilde{x}$ and $\tilde{y}$ be the opposite word of $x$ and $y$ respectively. Then 
    \[
    \beta = {\color{teal}{x}}s_3y s_1 s_2 s_1^{a_2} s_3 {\color{teal}s_2^{b_2} s_1^{a_3} s_2^{b_3}} \stackrel{\rho, c}{=} s_2^{b_2} s_1^{a_3} s_2^{b_3}  s_3  s_1 s_2 s_1^{a_2} ys_3x \stackrel{\textrm{oppo}}{\rightsquigarrow} \tilde{x}s_3\tilde{y}s_1^{a_2}s_2s_1s_3s_2^{b_3}s_1^{a_3}s_2^{b_2}.
    \]
 It goes back to Case A. 
 
 If $w_1(1,3) \succeq s_1^2s_3$, then 
    $$
    \beta
    \succeq s_3s_1{\color{blue}{s_1s_2s_1}}s_1s_3s_2^{b_2}s_1^{a_3}s_2^{b_3}
    \stackrel{\textrm{R3}}{=}
    s_3s_1s_2{\color{red}{s_1}}s_2s_1s_3{\color{red}{s_2^{b_2}s_1^{a_3}s_2^{b_3}}}
    {\succ} s_1s_3s_2^2s_1s_3s_2^2.
    $$
    
It remains to consider $w_1(1,3)\succeq s_1s_3^2$.  There are three possibilities for $w_1(1,4)$:
    $$
    w_1(1,4)\succeq \; 
    (e)\; s_1s_4s_3s_4s_3, \;
    (f)\; s_1s_3s_4s_3s_4, \;
    (g)\; s_1s_4s_3s_3s_4. 
    $$
For both ($e$) and ($f$), we have $w_1(1,4)\succ s_1s_3^3$. Then 
        \begin{align*}
         \beta 
         & \succ s_1{\color{blue}{s_3^3s_2s_3}}s_1^2s_2s_1^2s_2
        \stackrel{\textrm{R3}}{=}
        s_1s_2{\color{teal}{s_3}}s_2^3s_1^2s_2s_1^2s_2
        \stackrel{\textrm{R1}}{=}
        s_1s_2^4s_1^2{\color{teal}{s_2s_1^2s_2}}  \\
        & \stackrel{\rho}{=}
        s_2s_1^2s_2s_1s_2^4s_1^2
        \stackrel{\textrm{R3}}{=}
        s_2s_1^2s_1^4s_2s_1s_1^2
        =
        s_2s_1^6s_2s_1^3
        \end{align*}
        This is again the $\mathrm{E}_9$ quiver.
For ($g$), we have $w_1(1,4)\succeq s_1s_4s_3s_3s_4$. Then
        \begin{align*}
        \beta 
        &\succeq
        s_1{\color{teal}{s_4}}s_3^2{\color{teal}{s_4}}s_2s_1^2s_3s_2s_1^2s_2
        \stackrel{\textrm{c},\rho}{=}
        s_1s_3^2s_2s_1^2{\color{blue}{s_4s_3s_4}}s_2s_1^2s_2
        \stackrel{\textrm{R3}}{=}
        s_1s_3^2s_2s_1^2s_3{\color{red}{s_4}}s_3s_2{\color{red}{s_1^2}}s_2 \\
        &{\succ} 
        s_1s_3^2s_2s_1^2{\color{teal}{s_3}}s_3s_2s_2
        \stackrel{\textrm{c}}{=}
        s_1s_3{\color{blue}{s_3s_2s_3}}s_1^2s_3s_2s_2
        \stackrel{\textrm{R3}}{=}
        s_1s_3s_2{\color{red}{s_3}}s_2{\color{red}{s_1}}s_1s_3s_2s_2 \\
        &{\succ} 
        s_1s_3s_2^2s_1s_3s_2^2.
        \end{align*}

We complete the proof of Proposition \ref{Mainquiver}.

\begin{cor} For positive braids $[\beta]$ with connected $Q_\beta$, the two cases in Proposition \ref{Mainquiver} coincides with the dichotomy between finite and infinite types for positive braids.
\end{cor}
\begin{proof} It follows from Proposition \ref{Mainquiver} and Proposition \ref{3.4}.
\end{proof}

\noindent \textit{Proof of Theorem \ref{MainStep1} for disconnected $Q_\beta$.} Suppose $Q_\beta$ has two components. Because vertices on the same level are connected, there exists a unique $1\leq i<n$ such that no arrow appears between level $i$ and $i+1$. We consider $\beta(1,i)$ and $\beta(i+1,n-1)$. Since we can pinch some crossings of $\beta$ to obtain $\beta(1,i)$ and $\beta(i+1,n-1)$, if one of them has infinitely many 
admissible fillings, so does $\beta$ by Proposition \ref{3.2}. Otherwise by Propositions \ref{Mainquiver} (1) and \ref{3.4}, both $Q_{\beta(1,i)}$ and $Q_{\beta(i+1,n-1)}$ are mutation equivalent to finite type quivers, and hence $[\beta]$ is of finite type. In general, we can induct on the number of components in the quiver of the braid. \qed

\section{Finite Type Classification} \label{sec 4}

In this section, we focus on positive braid Legendrian links  of finite type.

\begin{thm}\label{Dynkin quiver} Let $\beta$ be a braid word such that $Q_\beta$ is mutation equivalent to a Dynkin quiver and $\Lambda_\beta$ does not contain a split union of knots. Then $\Lambda_\beta$ is Legendrian isotopic to a standard link in Definition \ref{stad.links}.
\end{thm}
\begin{proof} By Proposition \ref{Mainquiver} (1), it suffices to assume that $Q_\beta$ is a  Dynkin quiver. 

If $Q_\beta$ is of type $\mathrm{A}$, we repeated utilize Lemma \ref{3.13} to reduce the number of strands of $\Lambda_\beta$ until it becomes $2$-strand link, which is a standard link of type A.

If $Q_\beta$ is of type $\mathrm{D}$ or $\mathrm{E}$, then it contains  a unique trivalent vertex.  If $n\geq 4$, we can apply Lemma \ref{3.13} to $\beta(1,2)$ or $\beta(n-2,n-1)$, whichever does not contain the trivalent vertex, to reduce $n$ until $n=3$. Note that $\beta$ can be written as \eqref{beta.exp}. Since $[\beta]$ is of finite type, following the discussion in Section \ref{sec.3.2}, we may assume $m=2$ in \eqref{beta.exp}. After necessary rotation, we get
\[
\beta= s_1^{a_1}s_2^{b_1}s_1^{a_2}s_2^{b_2}, \quad \quad \mbox{where } a_1\geq 2, ~a_2 \geq 2, ~ \min\{b_1, b_2\}=1.
\]

 The trivalent vertex in a Dynkin $\mathrm{DE}$ quiver has three legs, at least one of which is of length $1$. For $Q_\beta$, two legs lie in level $1$ and one leg stretches to level $2$. We show that $b_1=b_2= 1$ after suitable Legendrian isotopy. Otherwise, one of the level $1$ legs is of length $1$. Then up to cyclic rotations, we get $a_2=2$. Depending on $b_1=1$ or $b_2=1$, we have the following Legendrian isotopies:
\begin{align*}
    \beta &= 
    s_1^{a_1}s_2s_1^2s_2^{b_2}
    =
    s_1^{{a_1}-1}{\color{blue}{s_1s_2s_1}}s_1s_2^{b_2}
    \stackrel{\textrm{R3}}{=}
    s_1^{a_1-1}s_2s_1{\color{blue}{s_2s_1s_2^{b_2}}}
    \stackrel{\textrm{R3}}{=}
    s_1^{a_1-1}s_2s_1^{b_2+1}s_2{\color{teal}{s_1}}
    \stackrel{\rho}{=}
    s_1^{a_1}s_2s_1^{b_2+1}s_2,    \\
    \beta &= 
    s_1^{a_1}s_2^{b_1}s_1^2s_2
    =
    {\color{teal}{s_1}}s_1^{{a_1}-1}s_2^{b_1}s_1^2s_2
    \stackrel{\rho}{=}
    s_1^{a_1-1}s_2^{b_1}s_1{\color{blue}{s_1s_2s_1}}
    \stackrel{\textrm{R3}}{=}
    s_1^{a_1-1}{\color{blue}{s_2^{b_1}s_1s_2}}s_1s_2
    \stackrel{\textrm{R3}}{=}
    s_1^{a_1}s_2s_1^{b_1+1}s_2.
\end{align*}
Eventually, after necessary cyclic notations, we get the standard links. 
\end{proof}

\begin{defn} Let $\beta$ be an $n$-strand braid word and let $\gamma$ be an $m$-strand braid word. Denote by $\gamma^{\#_j}$ the word obtained from $\gamma$ via $s_i\mapsto s_{i+j}$. 

The \emph{connect sum} of $\beta$ and $\gamma$ is the braid word $\beta \# \gamma:=\beta\gamma^{\#_{n-1}}$.

The  \emph{split union} of $\beta$ and $\gamma$ is the braid word $\beta \sqcup \gamma:=\beta\gamma^{\#_{n}}$.

Note that $\left[\beta \# \gamma\right] \in \Br_{n+m-1}^+$ and $[\beta\sqcup\gamma] \in \Br_{n+m}^+$.
\end{defn}

The connect sum of two positive braid links is again a positive braid link. By \cite{EV}, positive braid links attain a unique maximum tb Legendrian representative. The connect sum of two links is well-defined once specifying which components to attach the $1$-handle. Once well-defined, the connect sum is associative and commutative.

\begin{rmk} 
Below is  a list of the numbers of components for the standard $\mathrm{ADE}$ links.
\begin{center}
    \begin{tabular}{|c|c|c|} \hline
        \rule{0pt}{2.3ex} knots & 2-component links & 3-component links  \\[0.045cm] \hline
        \rule{0pt}{2.3ex} $\mathrm{A}_\text{even}$, $\mathrm{E}_6$, $\mathrm{E}_8$ & $\mathrm{A}_\text{odd}$, $\mathrm{D}_\text{odd}$, $\mathrm{E}_7$ & $\mathrm{D}_\text{even}$\\[0.045cm] \hline
    \end{tabular}
\end{center}
\end{rmk}

\begin{thm}\label{4.6} If $\beta$ is of finite type, then $\Lambda_\beta$ is Legendrian isotopic to a split union of unknots and connect sum of standard $\mathrm{ADE}$ links.
\end{thm}

\begin{proof}  The vertices on each level of $Q_\beta$ form a type $\mathrm{A}$ quiver.
If $Q_\beta$ is disconnected, then  we have
\begin{enumerate}
\item two adjacent levels of $Q_\beta$ have vertices but no arrows in between; and/or  
\item a level of $Q_\beta$ has no vertex.
\end{enumerate}

For (1), after necessary rotation, we get  $\beta(i, i+1) = s_i^a s_{i+1}^b$ for some $i$. We may further commute $s_1,\dotsb, s_{i-1}$ with $s_{i+2},\dotsb, s_{n-1}$, obtaining \[\beta = \beta(1,i) \beta(i+1,n-1).\] Hence, $\beta$ is a connect sum of two braid words. 

For (2), we get $\beta(i,i) = s_i$ or empty  for some $i$. If it is empty, then 
\[\beta = \beta(1,i-1) \beta(i+1,n),\] which is a split union of two braid words. If  $\beta(i,i) = s_i$, then the braid is a connect sum via the following Legendrian isotopy:
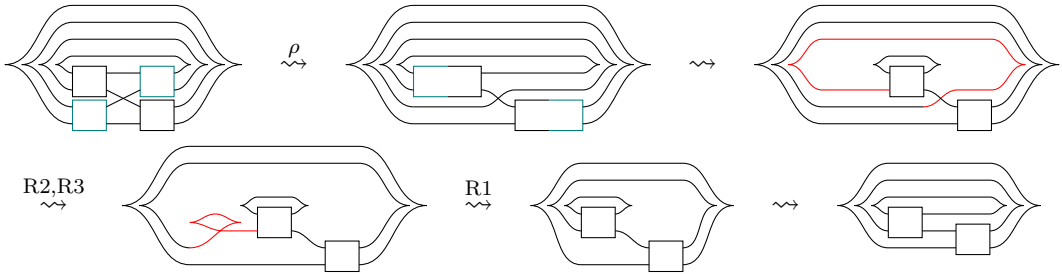
\[
\begin{tikzpicture}[baseline=20,scale=0.9]
\draw [teal](0,-0.1) rectangle (0.5,0.35);
\draw (0,0.4) rectangle (0.5,0.85);
\draw (1,-0.1) rectangle (1.5,0.35);
\draw [teal](1,0.4) rectangle (1.5,0.85);
\draw (0.5,0.25) -- (1,0.5);
\draw (0.5,0.5) -- (1,0.25);
\draw (0.5,0.75) -- (1,0.75);
\draw (0.5,0) -- (1,0);
\draw (1.5,0.75) to [out=0,in=180] (1.75,0.875) to [out=180,in=0] (1.5,1)--(0,1) to [out=180,in=0] (-0.25,0.875) to [out=0,in=180] (0,0.75);
\draw (1.5,0.5) to [out=0,in=180] (2,0.875) to [out=180,in=0] (1.5,1.25) -- (0,1.25) to [out=180,in=0] (-0.5,0.875) to [out=0,in=180] (0,0.5);
\draw (1.5,0.25) to [out=0,in=180] (2.25,0.875) to [out=180,in=0] (1.5,1.5) -- (0,1.5) to [out=180,in=0] (-0.75,0.875) to [out=0,in=180] (0,0.25);
\draw (1.5,0) to [out=0,in=180] (2.5,0.875) to [out=180,in=0] (1.5,1.75) -- (0,1.75) to [out=180,in=0] (-1,0.875) to [out=0,in=180] (0,0);
\end{tikzpicture} \quad \overset{\rho}{\rightsquigarrow} \quad 
\begin{tikzpicture}[baseline=20,scale=0.9]
\draw [teal] (2,-0.1) -- (2.5,-0.1) -- (2.5,0.35) -- (2,0.35);
\draw (2,-0.1) -- (1.5,-0.1) -- (1.5,0.35) -- (2,0.35);
\draw (0.5,0.4) -- (1,0.4) -- (1,0.85) -- (0.5,0.85);
\draw [teal] (0.5,0.4) -- (0,0.4) -- (0,0.85) --  (0.5,0.85);
\draw (1,0.75) -- (2.5,0.75) to [out=0,in=180] (2.75,0.875) to [out=180,in=0] (2.5,1)--(0,1) to [out=180,in=0] (-0.25,0.875) to [out=0,in=180] (0,0.75);
\draw (0,0.25) -- (1,0.25) to [out=0,in=180] (1.5,0.5) -- (2.5,0.5) to [out=0,in=180] (3,0.875) to [out=180,in=0] (2.5,1.25) -- (0,1.25) to [out=180,in=0] (-0.5,0.875) to [out=0,in=180] (0,0.5);
\draw (1,0.5) to [out=0,in=180] (1.5,0.25);
\draw (2.5,0.25) to [out=0,in=180] (3.25,0.875) to [out=180,in=0] (2.5,1.5) -- (0,1.5) to [out=180,in=0] (-0.75,0.875) to [out=0,in=180]  (0,0.25);
\draw (2.5,0) to [out=0,in=180] (3.5,0.875) to [out=180,in=0] (2.5,1.75) -- (0,1.75) to [out=180,in=0] (-1,0.875) to [out=0,in=180] (0,0) -- (1.5,0);
\end{tikzpicture} \quad \rightsquigarrow\quad 
\begin{tikzpicture}[baseline=20,scale=0.9]
\draw (0,0.4) rectangle (0.5,0.85);
\draw (1,-0.1) rectangle (1.5,0.35);
\draw (0.5,0.75) to [out=0,in=180] (0.75,0.875) to [out=180,in=0] (0.5,1) -- (0,1) to [out=180,in=0] (-0.25,0.875) to [out=0,in=180] (0,0.75);
\draw (1.5,0) to [out=0,in=180] (2.5,0.875) to [out=180,in=0] (1.5,1.75) -- (-1,1.75) to [out=180,in=0] (-2,0.875) to [out=0,in=180] (-1,0) --  (1,0);
\draw [red] (0.5,0.25) to [out=0,in=180] (1,0.5) -- (1.5,0.5) to [out=0,in=180] (2,0.875) to [out=180,in=0] (1.5,1.25) -- (-1,1.25) to [out=180,in=0] (-1.5,0.875) to [out=0,in=180] (-1,0.5) -- (0,0.5);
\draw (0.5,0.5) to [out=0,in=180] (1,0.25);
\draw (1.5,0.25) to [out=0,in=180] (2.25,0.875) to [out=180,in=0] (1.5,1.5) -- (-1,1.5) to [out=180,in=0] (-1.75,0.875) to [out=0,in=180] (-1,0.25) -- (0.5,0.25);
\end{tikzpicture}
\]
\[
\overset{\text{R2,R3}}{\rightsquigarrow}\quad
\begin{tikzpicture}[baseline=20,scale=0.9]
\draw (0,0.4) rectangle (0.5,0.85);
\draw (1,-0.1) rectangle (1.5,0.35);
\draw (0.5,0.75) to [out=0,in=180] (0.75,0.875) to [out=180,in=0] (0.5,1) -- (0,1) to [out=180,in=0] (-0.25,0.875) to [out=0,in=180] (0,0.75);
\draw (1.5,0) to [out=0,in=180] (2.5,0.875) to [out=180,in=0] (1.5,1.75) -- (-1,1.75) to [out=180,in=0] (-2,0.875) to [out=0,in=180] (-1,0) -- (1,0);
\draw [red] (-1,0.25) to [out=0,in=180] (-0.25,0.625) to [out=180,in=0] (-0.625,0.75) to [out=180,in=0] (-1,0.625) to[out=0,in=180] (-0.5,0.5) -- (0,0.5);
\draw (0.5,0.5) to [out=0,in=180] (1,0.25);
\draw (1.5,0.25) to [out=0,in=180] (2.25,0.875) to [out=180,in=0] (1.5,1.5) -- (-1,1.5) to [out=180,in=0] (-1.75,0.875) to [out=0,in=180] (-1,0.25);
\end{tikzpicture}\quad 
\overset{\text{R1}}{\rightsquigarrow}\quad 
\begin{tikzpicture}[baseline=20,scale=0.9]
\draw (0,0.4) rectangle (0.5,0.85);
\draw (1,-0.1) rectangle (1.5,0.35);
\draw (0.5,0.5) to [out=0,in=180] (1,0.25);
\draw (0.5,0.75)  to [out=0,in=180] (0.75,0.875) to [out=180,in=0] (0.5,1) -- (0,1) to [out=180,in=0] (-0.25,0.875) to [out=0,in=180] (0,0.75);
\draw (1.5,0.25) to [out=0,in=180] (2,0.875) to [out=180,in=0] (1.5,1.25) -- (0,1.25) to [out=180,in=0] (-0.5,0.875) to [out=0,in=180] (0,0.5);
\draw (1.5,0) to [out=0,in=180] (2.25,0.875) to [out=180,in=0] (1.5,1.5) -- (0,1.5) to [out=180,in=0] (-0.75,0.875) to [out=0,in=180] (0,0) -- (1,0);
\end{tikzpicture}
\quad \rightsquigarrow\quad
\begin{tikzpicture}[baseline=20,scale=0.9]
\draw (0,0.4) rectangle (0.5,0.85);
\draw (1,0.15) rectangle (1.5,0.6);
\draw (0.5,0.5) -- (1,0.5);
\draw (0.5,0.75) -- (1.5,0.75) to [out=0,in=180] (1.75,0.875) to [out=180,in=0] (1.5,1) -- (0,1) to [out=180,in=0] (-0.25,0.875) to [out=0,in=180] (0,0.75);
\draw (1.5,0.5) to [out=0,in=180] (2,0.875) to [out=180,in=0] (1.5,1.25) -- (0,1.25) to [out=180,in=0] (-0.5,0.875) to [out=0,in=180] (0,0.5);
\draw (1.5,0.25) to [out=0,in=180] (2.25,0.875) to [out=180,in=0] (1.5,1.5) -- (0,1.5) to [out=180,in=0] (-0.75,0.875) to [out=0,in=180] (0,0.25)--(1,0.25);
\end{tikzpicture}
\]
Each quiver component is Legendrian isotopic to the standard $\mathrm{ADE}$ links. There could also be a split union of unknot for every pair of consecutive levels $\beta(i,i)$ and $\beta(i+1,i+1)$ that are both empty. We complete the proof.
\end{proof}

\section{Applications} \label{sec.5.app}

The infinitely many fillings arising from the aperiodic DT transformation imply that the self-concordance monoid and the fundamental group of the space of Legendrian embeddings are infinite.

\begin{cor} For any braid word  $\beta$, if the \emph{DT} transformation for $Q_\beta$ is aperiodic, then
\begin{enumerate}[label*=\emph{(\arabic*)}]
    \item the Lagrangian self-concordance monoid $\text{Con}(\Lambda_\beta)$ has a subgroup $\bbZ$;
    \item the fundamental group of Legendrian embeddings  $\pi_1\mathcal{L} eg(\Lambda_\beta)$ has a subgroup $\bbZ$.
\end{enumerate}{}
\end{cor}
\begin{proof}
(1) Following the convention in the proof of Theorem \ref{5.11}, if the graph of $\cR^m$ is Hamiltonian isotopic to the trivial cylinder, then its $L_m$ and $L_0$ would induce the same chart, but they do not by Theorem \ref{5.11}.

(2) Consider the monoid morphism 
$\pi_1\mathcal{L} eg(\Lambda_\beta) \rightarrow \text{Con}(\Lambda_\beta).$
Two loops are distinct if they graph distinct concordances.
\end{proof}

\begin{rmk}
The proof of Corollary \ref{rainbowinfinite} implies that the Lagrangian self-concordance monoid of $\Lambda_\beta$ contains a $\bbZ$-subgroup. Meanwhile, the results in \cite{CasalsGao,CZ} manifest additional strength in those cases. In \cite{CasalsGao}, the Lagrangian self-concordance monoid of $\Lambda_{(3,6)}$ has a factor of $\text{PSL}(2,\bbZ)$, and that of $\Lambda_{(4,4)}$ has a factor of the mapping class group $M_{0,4}$ (spherical braid group in $4$-strands mod center), yielding the monoid of each link has a subgroup of exponential-growth. The construction in \cite{CZ} uses Legendrian weaves instead of Legendrian loops, which can be extended beyond positive braid Legendrian links. Also, the Lagrangian disks in the Polterovich surgery can be visualized explicitly.
\end{rmk}

The infinitely many fillings can be used to construct Weinstein $4$-manifolds or Stein surfaces that admit infinitely many closed exact Lagrangian surfaces. Let $\Lambda\subset \bbR^3\subset S^3$ be a Legendrian knot with infinitely many exact Lagrangian fillings. Attaching a Weinstein $2$-handle along $\Lambda$, we obtain a Weinstein manifold \cite{weinstein1991}, which is a Stein surface with the homotopy type of $S^2$ \cite{Eliashberg90,Gompf}. Using this construction, it was first shown in \cite{CasalsGao} that there exists a Stein surface which is \emph{homotopic to the $2$-sphere} and has infinitely many closed exact Lagrangian surfaces of \emph{higher genus}. These Lagrangian surface are not related by symplectic Dehn twists \cite{Arnoldmilnor,Seidelknotted}.

Prior to the emergence of the infinite filling approach, closed exact Lagrangian surfaces were constructed in some cases.
\begin{enumerate}[leftmargin=*]
    \item Seidel proved infinitely many Lagrangian $2$-spheres in $A_k$-Milnor fibers, $k\geq 3$ \cite{Seidelgraded}.
    \item Keating constructed an exact Lagrangian torus in the $A_{p,q,r}$-Milnor fiber, that cannot be expressed by $2$-spheres in the Fukaya category \cite{Keating}.
    \item Vianna constructed infinitely many exact Lagrangian tori in $\mathbb{C}\mathbb{P}^2\setminus \{\textrm{smooth cubic}\}$ \cite{vianna2014}. (The original result states that there are infinitely many monotone Lagrangian tori in $\mathbb{C}\mathbb{P}^2$, but one can delete the smoothing of the toric divisor to adapt to the exact setting.) Later the result was generalized to del Pezzo surfaces \cite{Viannadelpezzo}.
\end{enumerate}
The Weinstein manifold in (1) is homotopic to a bouquet of spheres. The Weinstein manifold in (2) or (3) is not homotopic to a bouquet of spheres, i.e. $\pi_1(\mathbb{C}\mathbb{P}^2\setminus \{\textrm{smooth cubic}\}) =\bbZ_3$.

By \cite{CasalsGao}, there are Weinstein manifolds with the homotopy type $S^2$ and infinitely many Lagrangian fillings of genus $g\geq 7$. Using the examples constructed in this section, we can lower the genus bound to $g\geq 4$.

\begin{cor}\label{application Weinstein}
For any $g\geq 4$, there is a Stein surface that is homotopic to $S^2$ and contains infinitely many exact Lagrangian surfaces of genus $g$ that are smoothly isotopic but non-Hamiltonian isotopic.
\end{cor}{}
\begin{proof}
By Lemma \ref{basicbraidsinf} (1), $\beta_0 = s_1^2s_2^2s_1^2s_2^2\in \Br_3^+$ has infinite fillings. Then $\beta = s_1s_2\beta_0$ gives rise to a Legendrian knot with tb $= 10-3 =7$, and then its fillings have genus $g=4$ \cite{chantraine2010}. Attach a a Weinstein $2$-handle along $\Lambda_{\beta}$. The concatenation of a Lagrangian filling with the core of the Weinstein handle produces a closed exact Lagrangian surface. These Lagrangian surfaces are distinct, following the same argument as in \cite[Corollary 1.10]{CasalsGao}. For higher $g$, consider $s_1^{2g-8}\beta$.
\end{proof}{}

\begin{rmk}
We do not know if there is a Weinstein manifold homotopic to $S^2$ with infinitely many exact Lagrangian surfaces of genus $g= 0,1,2$ or $3$.
\end{rmk}

We can lower the genus bound of the closed Lagrangian if we are willing to trade off the homotopy type of the Weinstein manifold. Consider Legendrian links with infinitely many Lagrangian fillings. After attaching a Weinstein $2$-handle at each component, we obtain a Weinstein manifold which is homotopic to a bouquet of spheres. The Weinstein structure does not depend on the order of the handle attachment. Previous methods have achieved $g\geq 4$ \cite{CasalsGao,CZ}. Now it is true for any genus.

\begin{cor}
For any $g\in \mathbb{N}$, there is a Stein surface that is homotopic to a bouquet of spheres, and contains infinitely many exact Lagrangian surfaces of genus g that are smoothly isotopic but non-Hamiltonian isotopic.  
\end{cor}

\begin{proof}
The case $g=0$ follows from \cite{Seidelgraded}. For $g\geq 1$, the braid $s_1^{2g-2}s_1^2s_2^2s_1^2s_2^2\in \Br_3^+$ gives a $4$-component Legendrian link with genus $g$.
\end{proof}

\bibliographystyle{alphaurl-a}

\bibliography{biblio}

\end{document}